\documentclass[11pt]{article}

\usepackage{amsthm}
\usepackage{amsmath}
\usepackage{amsfonts, epsfig, amsmath, amssymb, color, amscd}
\usepackage[cmtip,arrow]{xy}
\usepackage{pb-diagram, pb-xy}
\usepackage{amssymb,epsfig}
\usepackage{color}
\usepackage[cmtip,arrow]{xy}
\usepackage{pb-diagram, pb-xy}
\usepackage{amssymb,epsfig,amsfonts}

\newfont{\sdbl}{msbm9}
\newfont{\dbl}{msbm10 at 12pt}
\theoremstyle{definition}
\newcommand{\da}{{\mbox{\dbl A}}}

\newcommand{\cl}{{{\cal L}}}
\newcommand{\cg}{{{\cal G}}}
\newcommand{\dpp}{{\mbox{\dbl P}}}
\newcommand{\dz}{{\mbox{\dbl Z}}}

\newcommand{\dg}{{\mbox{\dbl G}}}

\newcommand{\dq}{{\mbox{\dbl Q}}}

\newcommand{\ord}{\mathop{\rm ord}\nolimits}

\newcommand{\Pic}{\mathop {\rm Pic}}

\newcommand{\Quot}{\mathop {\rm Quot}}

\newcommand{\rk}{\mathop {\rm rk}}

\newcommand{\Spec}{\mathop {\rm Spec}}

\newcommand{\Proj}{\mathop {\rm Proj}}

\newcommand\limind{\mathop{\underrightarrow{\lim}}}

\makeatother
\newtheorem{defin}{Definition}[section]
\newtheorem{rem}{Remark}[section]

\theoremstyle{plain}
\newtheorem{prop}{Proposition}[section]
\newtheorem{thm}{Theorem}[section]
\newtheorem{lem}{Lemma}[section]
\newtheorem{cor}{Corollary}[section]
\newtheorem{claim}{Claim}[section]
\newtheorem{property}{Property}
\newtheorem{que}{Question}

\def\Spec{{{\rm Spec \,}}}

\def\Proj{{{\rm Proj \,}}}

\def\Quot{{{\rm Quot \,}}}

\def\Pic{{{\rm Pic \,}}}

\def\dim{{{\rm dim \,}}}

\def\ker{{{\rm ker \,}}}

\newcommand{\eqdef}{\stackrel{\rm def}{=}}

\newcommand{\Ord}{\mathop {\rm \bf ord}}
\newcommand{\Rk}{\mathop {\rm \bf rk}}
\newcommand{\co}{{{\cal O}}}
\newcommand{\cf}{{{\cal F}}}

\newcommand{\cm}{{{\cal M}}}

\newcommand{\cb}{{{\cal B}}}

\title{Amazing	examples	of		nonrational
smooth	spectral	surfaces}
\author{Viktor Kulikov \quad Alexander Zheglov}

\date{}

\begin{document}

\maketitle

\quad \qquad \qquad
{\em Dedicated to A.~N.~Parshin on the occasion of his 75th birthday}

\begin{abstract}
In this paper we construct first examples of  smooth projective surfaces of general type satisfying the following conditions:
there are  1) an ample integral curve $C$ with $C^2=1$ and $h^0(X,\co_X(C))=1$; \quad 2) a divisor $D$ with $(D, C)_X=g(C)-1$, $h^i(X,\co_X(D))=0$, $i=0,1,2$, and $h^0(X,\co_X(D+C))=1$.

Such conditions arise from necessary and sufficient conditions for the existence of non-trivial commutative subalgebras of rank one in $\hat{D}$, a completion of the algebra of partial differential operators in two variables, which can be thought of as a simple algebraic analogue of the algebra of analytic pseudodifferential operators on a manifold. We extract these conditions by elaborating
the classification theorem of commutative subalgebras in $\hat{D}$ due to the second author for the case of rank one subalgebras. Amazingly, the commutative subalgebras with such spectral surfaces do not admit isospectral deformations.
\end{abstract}


\section{Introduction}

The  theory of commuting ordinary differential or difference operators appears in mathematical physics as an algebro-geometric tool in the theory of integrating non-linear soliton systems and the spectral theory of periodic finite-zone operators (see \cite{DMN1976}, \cite{DKN1985},  \cite{Krichever77}, \cite{KricheverNovikov2003}). An effective classification was offered in  \cite{Krichever77},  \cite{Krichever78} for differential case and in \cite{KricheverNovikov2003} for difference case (under certain natural restrictions). The theory of commuting partial differential operators or higher-dimensional difference operators is much more complicated and is not yet completed, though some elements of it appeared in the literature since the work \cite{Krichever77}. Namely, a lot of papers were devoted to the explicit constructions of commuting matrix differential operators, see e.g. \cite{Dubrovin83}, \cite{Grinevich86}, \cite{Nakayashiki91}, \cite{Nakayashiki94}, \cite{Mironov2002}, \cite{MironovNakayashiki2011}, \cite{Rothstein2004}. Non-trivial examples of commuting partial differential operators with {\it scalar} coefficients appeared in the deep theory of quantum Calogero-Moser systems and their deformations, see e.g. \cite{OP83}, \cite{HeckmanOpdam87}, \cite{Heckman91}, \cite{ChalykhVeselov90}, \cite{ChalykhVeselov93}, \cite{ChalykhVeselov98}, \cite{BerestEtigofGinzburg03}, \cite{FeiginJohnston14}, \cite{Ch} and references therein. Other examples obtained with the help of ideas from differential algebra or theory of $D$-modules see e.g. in \cite{Previato08}, \cite{BerestKasman98}.

Investigating the theory of commuting operators with scalar coefficients, the second author offered  in \cite{Zheglov2013} an analogue of the Krichever classification theorem for commutative subalgebras in a certain completion $\hat{D}$ of the algebra of partial differential operators in two variables. This completion can be thought of as a simple algebraic analogue of the algebra of analytic pseudodifferential operators on a manifold. It contains many important pseudodifferential operators that appear in diverse context in mathematical physics and analysis, among them are also difference operators, see section 2 (in particular, remark \ref{R:completion}). In the approach of \cite{Zheglov2013} the subalgebras in $\hat{D}$ appears quite naturally, in particular as isospectral deformations of subalgebras in $D$. The similarity of theories (the theory of commuting ordinary differential operators and the theory in dimension two) leads to a natural conjecture that subalgebras of rank one permit an easier description in algebro-geometric terms (than a description of higher rank algebras).  In particular, conjecturally there should be certain analogues of the explicit Krichever-Novikov formulas for rank one commuting ordinary differential or difference operators. This conjecture is justified by a number of explicit examples: first, by the examples arising in the theory of quantum Calogero-Moser systems (see the references above); second, by the examples obtained  with the help of the new approach from works \cite{Zheglov2013}, \cite{KurkeZheglov}, \cite{BurbanZheglov2017} (see \cite{ZheglovDiss} for the modern state of art). We start our paper with a statement extracting such easier description of rank one subalgebras from the general classification theorem in \cite{Zheglov2013} (first giving an appropriate definition of rank one subalgebras). Refining this result, we come to a simple notion of {\it pre-spectral data of rank one} which encodes rank one subalgebras in $\hat{D}$. This data consists of a projective surface $X$, ample $\dq$-Cartier divisor $C$ and a torsion free sheaf $\cf$ with a number of simple properties, see definition \ref{D:pre-spectral}.

Unlike the theory in dimension one, there are strong restrictions on the geometry of algebro-geometric spectral data of commutative subalgebras from $\hat{D}$.  In particular, in all examples from \cite{Zheglov2013}, \cite{KurkeZheglov}, \cite{BurbanZheglov2017} the spectral surfaces are {\it singular}. First parameter that divides the set of all possible spectral surfaces in two classes is the dimension of the following cohomology group: $h^0(X,\co_X(C))$. To explain its meaning, recall that there is a big class of rank one algebras which were called "trivial" in \cite{KurkeZheglov}. They are the algebras containing the operator $\partial_1$ or $\partial_2$, i.e. they consist of operators do not depending on $x_1$ or $x_2$. The examples of such algebras naturally arise from examples of commuting ordinary differential operators just by adding one extra derivation. The pre-spectral data corresponding to such subalgebras are characterised by the property $h^0(X,\co_X(C))\ge 2$, see section 4 (cf. \cite[Th.4.1]{KurkeZheglov}). In fact, the "trivial" algebras are not so trivial; investigation of their properties as well as of properties of their spectral data will appear in a separate paper.
First examples of explicit non-"trivial" algebras, which are deformations of well known rational Calogero-Moser systems (with singular spectral surface), obtained with the help of our classification theory appeared in \cite[\S 6]{BurbanZheglov2017}.

In this paper we investigate the question whether there are pre-spectral data of rank one with a {\it smooth} spectral surface {\it and} $h^0(X,\co_X(C))=1$. We construct first examples of such data with minimal possible genus of the divisor $C$; the surfaces are of general type. 

We leave for further investigations the following questions. 1) It would be interesting to find all possible smooth surfaces with such properties. We conjecture that for all such surfaces $q=p_g=0$. 2) It would be also interesting to find explicit commuting operators corresponding to such surfaces. 3) Other questions see at the end of the paper.

\bigskip

The paper is organized as follows. In the second section we give a review of the classification theory of commuting operators. At the end of section we prove a refined version of the classification theorem for commutative algebras of rank one.

In the third section we introduce the notion of pre-spectral data and show that it can be extended to a spectral data, thus reducing the problem of finding examples of commutative subalgebras in the algebra $\hat{D}$ to a purely algebro-geometric problem.

In the fourth section we investigate the question of existing {\it smooth spectral surfaces}. We recall the notion of "trivial" commutative subalgebras and show that smooth spectral surfaces with ample divisors of arithmetical genus less or equal to one lead to "trivial" commutative subalgebras. Then we show the existence of pre-spectral data with smooth spectral surfaces (of general type) and with smooth ample divisors of genus two. Amazingly the moduli space of spectral sheaves on such surfaces is finite, i.e. the corresponding commutative subalgebras do not have any isospectral deformations.

\section{Survey on the algebro-geometric spectral data}

In this section we elaborate the theory of commutative subalgebras  from \cite{Zheglov2013}, \cite[Ch. 3]{ZheglovDiss} for  {\it rank one} subalgebras (we refer to \cite[Sec. 2]{KurkeZheglov} for a review of this theory in generic case). Our aim is to prove a refined version of the classification theorem for commutative algebras of rank one, which is motivated by the following observation. 

Starting with the algebra of partial differential operators (PDO for short) $D=k[[x_1,x_2]][\partial_1,\partial_2]$, where $k$ is an algebraically closed field of characteristic zero,  one can define a completion $\hat{D}$. In \cite{Zheglov2013} it was shown that commutative $k$-subalgebras $B\subset \hat{D}$ satisfying certain mild conditions are classified in terms of certain geometric spectral data, and this classification is, in a sense, a natural generalisation of the Krichever classification of commuting ordinary differential operators. The classification of commuting ordinary operators is especially simple for  subalgebras of rank one, because any rank one commutative subalgebra of ordinary differential operators is essentially determined (up to a linear change of variables) by purely geometric spectral data consisting of a projective curve, a regular point on the curve, and a coherent torsion free sheaf of rank one over the curve.
We are going to explain in this section that the classification of commutative subalgebras in $\hat{D}$ of {\it rank one} (for an appropriately defined notion of rank) possess the analogous property.

We denote by $M$ the unique maximal ideal in the ring $k[[x_1,x_2]]$ and by $\ord_M$ the order function (or the discrete valuation defined by $M$) on this ring:
$$
\ord_M(a)=\sup \{n|a\in (x_1,x_2)^n\}.
$$
\begin{defin}
Define the ring
$$
\hat{D}_1=\{ a=\sum_{q\ge 0} a_q\partial_1^q |a_q\in k[[x_1,x_2]] \quad \mbox{and} \quad \sup\{q-\ord_M(a_q)\}<\infty\}.
$$
Define the {\it completion} of $D$ as $\hat{D}=\hat{D}_1[\partial_2]$.

We define also a kind of "pseudo-differential" ring $\hat{E}_+= \hat{D}_1((\partial_2^{-1}))$.
\end{defin}
\begin{rem}
\label{R:completion}
This definition differs from definition in \cite{Zheglov2013}(our $\hat{D}$ is $\hat{D}\cap \Pi_1$ in the notation from \cite{Zheglov2013}; $\hat{E}_+$ also differs correspondingly). Though this definition is not "symmetric" (there is a symmetric version of completion in \cite[Def. 5.1]{BurbanZheglov2017}), it is well adapted for the classification of commutative subalgebras. In particular, there is an analogue of the Schur theory \cite[Ch. 3]{ZheglovDiss}.

We'll call the elements from $\hat{D}$, $\hat{E}_+$ as operators. Operators from $\hat{D}$ act linearly on the space of  functions $k[[x_1,x_2]]$. For $P\in \hat{D}$ we'll denote this action by $P\diamond f$ or by $P(f)$. The algebra $D$ of partial differential operators (PDOs)  is a dense subalgebra in $\hat{D}$ (with respect to an appropriate topology). It contains also $k$-linear endomorphisms of $k[[x_1,x_2]]$, Dirac's delta-functions, integral operators (see \cite[Ex. 5.4, 5.5]{BurbanZheglov2017}) and difference operators (embedded e.g. through the embedding defined in \cite{MironovMaul}).
\end{rem}

Following the exposition of \cite[Sec. 5]{BurbanZheglov2017}, we introduce an analogue of the order function on $D$:
\begin{defin}
For any element
$
P = \sum\limits_{k_1,k_2 \ge \underline{0}} a_{\underline{k}} \underline{\partial}^{\underline{k}} \in \hat{D},
$
where $\underline{k}=(k_1,k_2)$, $\underline{\partial}^{\underline{k}} = \partial_1^{k_1} \partial_2^{k_2}$, we define its {\it order} to be
\begin{equation}
\label{E:LastOrder}
\Ord (P) := \sup\bigl\{k_1+k_2 - \ord_M(a_{\underline{k}})  \bigr\} \in \dz\cup \{ -\infty\}
\end{equation}
(we set $\Ord (0)=-\infty$).
Thus, if $d=\Ord (P)$, then we have:
$$
\ord_M(a_{\underline{k}})\ge k_1+k_2-d \quad \mbox{for any} \; a_{\underline{k}}.
$$
Note that for a partial differential  operator with constant highest symbol the order of $P$ taken in the sense \eqref{E:LastOrder} coincides with the usual definition of the order of a differential operator. The order function $\Ord$ induces a filtration on $\hat{D}$ and all its subalgebras, and we will denote by $gr(\cdot)$ the corresponding associated graded algebras.

Let $P \in \hat{D}$. Then we have  uniquely determined  $\alpha_{\underline{k}, \underline{i}} \in k$ such that
\begin{equation}\label{E:expOperatorP}
P = \sum\limits_{k_1, k_2, x_1, x_2 \ge \, 0} \alpha_{\underline{k}, \underline{i}} \,  \underline{x}^{\underline{i}} \underline{\partial}^{\underline{k}}.
\end{equation}
For any $m \ge -d$ we put:
$$
P_m:= \sum\limits_{(i_1+i_2) - (k_1+k_2) = m} \alpha_{\underline{k}, \underline{i}} \,  \underline{x}^{\underline{i}} \underline{\partial}^{\underline{k}}
$$
to be the $m$-th \emph{homogeneous component} of $P$. Note that $\Ord (P_m) = -m$ and we have a decomposition
$
P = \sum\limits_{m=-d}^\infty P_m.
$

Finally,  $\sigma(P) := P_{-d}$ is  the \emph{symbol} of $P$. We say that $P \in \hat{D}$ is \emph{homogeneous} if $P = \sigma(P)$.
\end{defin}

Unlike the usual ring of PDOs, the ring $\hat{D}$ contains zero divisors (see e.g. \cite[Ex. 5.4]{BurbanZheglov2017}). By this reason the order function and symbols have a little bit weaker properties.
\begin{lem}
\label{L:ord-propeties}
The following properties hold:
\begin{itemize}
\item
$\Ord (P\cdot Q) \le \Ord (P)+\Ord (Q)$, and the equality holds iff $\sigma (P)\cdot \sigma (Q)\neq 0$.
\item
$\sigma (P\cdot Q)= \sigma (P)\cdot \sigma (Q)$, provided $\sigma (P)\cdot \sigma (Q)\neq 0$.
\end{itemize}
\end{lem}

The proof of this lemma for the ring $\hat{D}$ is contained in the proof of \cite[Th. 5.3]{BurbanZheglov2017}.

The classification theorem from \cite{Zheglov2013} deals with commutative subalgebras from $\hat{D}$ satisfying certain mild special conditions (1-quasi-elliptic strongly admissible). To explain them, we need to recall the notions of order functions and of growth condition from \cite{Zheglov2013}.

\begin{defin}
\label{defin3}
We say  that a non-zero operator $P\in \hat{D}$ {\it has $\Gamma$-order} $\ord_{\Gamma} (P)=(k,l)$ if $P=\sum_{s=0}^lp_s\partial_2^s$, where $p_s\in \hat{D}_1$, $p_l\in k[[x_1,x_2]][\partial_1]=D_1$, and $\ord (p_l)=k$ (here $\ord$ is the usual order in the ring of differential operators $D_1$). In this situation we say that the operator $P$ is monic if the highest coefficient of $p_l$ is $1$.

The notion of $\Gamma$-order extends obviously to the ring $\hat{E}_+$. There is also the order {\it function} $\ord_2$ defined on $\hat{E}_+$ as
$
\ord_2(P)=l
$
if $P=\sum_{s=-\infty}^lp_s\partial_2^s$. The coefficient $p_l$ is called {\it the highest term} and will be denoted by $HT_2(P)$ (as the term naturally associated with the function $\ord_2$).
\end{defin}

Both orders  behave like the $\Ord$-function, namely
\begin{equation}
\label{E:order-properties}
\ord_{\Gamma}(P_1\cdot P_2)=\ord_{\Gamma}(P_1)+\ord_{\Gamma}(P_2); \quad \ord_2(P_1\cdot P_2)\le \ord_2(P_1)+\ord_2(P_2),
\end{equation}
see \cite[Lemma 2.8]{Zheglov2013} or \cite[S.3.3.1, L.14]{ZheglovDiss} for the proof of the first equality; the second inequality is obvious, moreover, we have $HT_2(P_1\cdot P_2)=HT_2(P_1)\cdot HT_2(P_2)$ provided $HT_2(P_1)\cdot HT_2(P_2)\neq 0$.

\begin{defin}
We say that an operator $Q=\sum q_{ij}\partial_1^{i}\partial_2^{j}\in \hat{E}_+$ {\it satisfies condition $A_{1}(m)$} if $\ord_M(q_{ij})\ge i+j-m$ for all $(i,j)$.

An operator $P\in \hat{D}$, $P=\sum p_{ij}\partial_1^{i}\partial_2^{j}$ with $\ord_{\Gamma} (P)=(k,l)$ {\it satisfies  condition $A_{1}$} if it satisfies $A_1(k+l)$.
\end{defin}

It is easy to see that the ring $\hat{D}$ consists of operators satisfying conditions $A_1(m)$ for some $m$. The subset
\begin{equation}
\label{F:Pi}
\Pi =\{P\in \hat{E}_+| \mbox{\quad $\exists $  $m\in \dz_+$ s. that $P$ satisfies $A_{1}(m)$}\} .
\end{equation}
is an associative subring with unity (see \cite[Corol.2.2]{Zheglov2013}).
The order function $\Ord$ extends obviously to $\Pi$ and this extension has the same properties from lemma \ref{L:ord-propeties}, as the same arguments show.

\begin{defin}
\label{D:elliptic}
The ring $B\subset \hat{D}$ of commuting operators is called quasi elliptic if it contains two monic operators $P,Q$ such that $\ord_{\Gamma} (P)= (0,k)$  and $\ord_{\Gamma} (Q)=(1,l)$ for some $k,l\in \dz_+$, $k\ge 1$.

The ring $B$ is called $1$-quasi elliptic if $P,Q$ satisfy the condition $A_{1}$. In this case  $\Ord (P)=k$, $\Ord (Q)=1+l$.
\end{defin}

The main features of $1$-quasi elliptic rings are that they are integral (see \cite[Th. 3.2]{Zheglov2013}) and are appropriate for the Schur theory \cite[Ch. 3]{ZheglovDiss}. The $\Gamma$-order is defined on all elements of such rings (see \cite[Lemma 2.3]{KurkeZheglov}), and thus the function $-\ord_{\Gamma}$ is a discrete valuation of rank two (i.e. a valuation with values in the group $\dz\oplus \dz$ with the anti-lexicographical order).
Moreover these rings satisfy the following important property (cf. \cite[Th. 2.1]{KOZ2014}).

\begin{lem}
\label{L:property}
Let $B$ be a $1$-quasi elliptic commutative subalgebra in $\hat{D}$. Then the natural map
$$
\Phi: gr(\hat{D}) \rightarrow gr(\hat{D})/x_1 gr(\hat{D})+x_2gr(\hat{D})\simeq k[\xi_1,\xi_2]
$$
induces an embedding of vector spaces on $gr(B)$. The $B$-module $F=\hat{D}/x_1\hat{D}+x_2\hat{D}\simeq k[\partial_1, \partial_2]$ (the {\it spectral module}) is torsion free.

In particular, the function $-\Ord$ induces a discrete valuation of rank one on $B$ and on its field of fractions $\Quot (B)$. Moreover, $\Ord (P)=k+l$, where $(k,l)=\ord_{\Gamma}(\sigma(P))$, $P\in B$.
\end{lem}

\begin{proof}
By \cite[Lemmas 2.9, 2.10, 2.11]{Zheglov2013} together with \cite[Corol. 3.1]{Zheglov2013} there exists an invertible operator $S\in \hat{E}_+$ satisfying the condition $A_{1}$ with $\ord_2(S)=0$ and with invertible $HT_2(S)\in \hat{D}_1$ such that $SBS^{-1}\subset \hat{E}_+$ is a subalgebra of operators with {\it constant} coefficients (in particular, $S\in \Pi$). More precisely, $S$ has the form $S=fS_1S_2$, where $f\in k[[x_1,x_2]]$ is an invertible function, $S_1\in \hat{D}_1$ is an invertible operator such that $[\partial_1,S_1]=0$, and $S_2=1+S_2^-$ with $S_2^-\in \hat{D}_1[[\partial_2^{-1}]]\partial_2^{-1}$.

Since $S$ is invertible, for any non-zero element $Q\in B$ we have $\Ord (SQS^{-1})=\Ord (Q)$. But the symbol of $SQS^{-1}$ has constant coefficients, therefore, the image $\sigma'(Q)$ of $\sigma (Q)$ in $k[\partial_1,\partial_2]$ is not zero.
Indeed, on the one hand side we have $HT_2(\sigma (S)\cdot \sigma (Q)\cdot \sigma (S)^{-1})=HT_2(\sigma (S))\cdot HT_2(\sigma (Q))\cdot HT_2(\sigma (S)^{-1})$. On the other hand, 
$HT_2(\sigma (S)\cdot \sigma (Q)\cdot \sigma (S)^{-1})=c_1\partial_1^m$, $\sigma (S)=c_2\cdot \sigma (S_1)\cdot \sigma (S_2)$ for some $c_i\in k^*$, and $HT_2(\sigma (S))=c_2\cdot \sigma (S_1)$. Thus, $HT_2(\sigma (Q))=c_2^{-1}\cdot \sigma (S_1)^{-1}\cdot c_1\partial_1^m \cdot c_2 \cdot \sigma (S_1)=c_1\partial_1^m $, i.e. $\sigma '(Q)\neq 0$ in $k[\partial_1,\partial_2]$. Thus, $gr(B)$ is embedded in $k[\partial_1,\partial_2]$ and $F$ is a torsion free $B$-module, as for any non-zero $f\in F$ $f\cdot S\neq 0$ and for any non-zero $b\in B$ $f\cdot b=fS^{-1}(SbS^{-1})S\neq 0 \mbox{\quad mod\quad } (x_1,x_2)$.
\end{proof}

For some $1$-quasi-elliptic rings it is possible to define a notion of rank.

\begin{defin}
\label{D:rank}
For a $1$-quasi-elliptic commutative ring $B\subset \hat{D}$ we define numbers $\tilde{N}_B$ and $N_B$ as
$$
N_B=GCD\{\ord_2(b):\quad b\in B, \quad \ord_{\Gamma}(b)=(0,\ord_2(b)), \quad \Ord (b)=\ord_2(b)\},
$$
$$
\quad \tilde{N}_B=GCD\{\Ord (b)|\quad b\in B\}
$$
We will say that the ring $B$ is {\it strongly admissible} if $\tilde{N}_B=N_B$. In this case we define the rank of $B$ as
$\rk(B)=N_B=\tilde{N}_B$.
\end{defin}

This definition is equivalent to \cite[Def. 3.5-3.8]{Zheglov2013} due to \cite[Th. 3.2]{Zheglov2013}. One of motivations to define $1$-quasi-elliptic strongly admissible rings was the following property of a sufficiently large class of commutative rings of PDEs (\cite[Lemma 2.6, Prop. 2.4]{Zheglov2013}): if $B$ contains two operators $P,Q$ such that the function $\sigma (P)^{\Ord Q}/\sigma (Q)^{\Ord P}$ is not a constant on $\dpp^1$, then after almost all linear changes of variables the ring $B$ becomes $1$-quasi-elliptic strongly admissible.

According to \cite[Th. 3.4]{Zheglov2013}, finitely generated commutative $1$-quasi-elliptic strongly admissible rings of rank $r$ are classified in terms of {\it geometric spectral data of rank $r$}  (we'll give below a simplified precise definition of  rank one data), which consists, in particular, of a projective surface (usually very singular), an ample irreducible $\dq$-Cartier divisor, a regular point on the divisor and on the surface, and a torsion free sheaf with some cohomological properties.

These geometric objects can be easily described: for a given ring $B$ define $\tilde{B}:=\bigoplus\limits_{n=0}^{\infty}B_ns^n$ to be the Rees algebra defined with respect to the filtration defined by $\Ord$. Then the {\it spectral surface} $X=\Proj \tilde{B}$, the {\it divisor} $C\simeq\Proj (gr (B))$ corresponds to the valuation $-\Ord$, the {\it spectral sheaf} $\cf =\Proj \tilde{F}$, where $\tilde{F}:=\bigoplus\limits_{n=0}^{\infty}F_ns^n$ is the Rees module defined with respect to the filtration defined by $\Ord$ and $\Proj$ means the sheaf associated with the graded module, and the {\it point} $p$ is the center of a discrete valuation of rank two $\nu_B$ defined as  $\nu_B (P):=(k,-\Ord)$, where $(k,l)=\ord_{\Gamma}(\sigma (P))$ for $P\in B$ (see the proof of theorem \ref{T:classification} below).

The spectral sheaf  plays the crusial role in this classification; it has the following general properties:
\begin{itemize}
\item
it is quasi-coherent, but may be not coherent \cite[Ex. 3.4]{KOZ2014};
\item
if it is coherent, its rank is greater or equal to the rank of the ring (=the rank of the data) \cite[Rem. 3.3]{KOZ2014} (there are some sufficient conditions on the ring of PDO's that guarantees the coherence of the spectral sheaf, see \cite[Prop. 3.3]{KOZ2014} combined with  \cite[Th. 2.1]{KOZ2014} or \cite[Th.18]{ZheglovDiss} for a version with corrected inaccuracies);
\item
it is coherent of rank {\it equal } to the rank of the ring iff the self-intersection index of the ample divisor is equal to the rank of ring \cite[Prop. 3.2]{KOZ2014}.
\end{itemize}

In case when the $B$-module $F$ from lemma \ref{L:property} is finitely generated, it has the following  natural interpretation, analogous to the case of rings of PDOs.

\begin{prop}
\label{P:spectral_module}
Let $B\subset \hat{D}$ be a finitely generated $1$-quasi-elliptic commutative subalgebra such that the module $F$ is finitely generated.

For any character $B \stackrel{\chi}\rightarrow k$ (i.e.~an algebra homomorphism), consider the vector space
\begin{equation}\label{E:solspace}
\mathsf{Sol}\bigl(B, \chi\bigr):= \Bigl\{f\in k[[ x_1, x_2]] \big| P\diamond f = \chi(P) f \; \mbox{\rm for all}\; P \in B\Bigr\}.
\end{equation}
Then there exists  a canonical isomorphism of vector spaces
\begin{equation}\label{E:solspaceisom}
F\big|_{\chi} := F \otimes_{B} \bigl(B/\mathrm{Ker}(\chi)\bigr) \cong \mathsf{Sol}\bigl(B, \chi\bigr)^\ast.
\end{equation}
assigning to a class $\overline{\partial_1^{p_1} \partial_2^{p_2}} \in F\big|_{\chi}$ the linear functional $f  \mapsto \left.\dfrac{1}{p_1! p_2!} \dfrac{\partial^{p_1+p_2}f}{\partial x_1^{p_1} \partial x_2^{p_2}}\right|_{(0, 0)}$ on the vector space $\mathsf{Sol}\bigl(B, \chi\bigr)$. In particular, $\dim_{k}\Bigl(\mathsf{Sol}\bigl(B, \chi\bigr)\Bigr)<\infty$  for any  $\chi$.
\end{prop}

The proof is verbally the same as in \cite[Th. 4.5 item 2]{BurbanZheglov2017}, cf. \cite[Rem. 2.3]{KOZ2014}.

Proposition \ref{P:spectral_module} permits to define another version of rank of a commutative subalgebra $B\subset \hat{D}$: let's denote by $\Rk (B)$ the rank of the module $F$. Having in mind the properties of the spectral sheaf mentioned above, we can now define what we'll mean as rank one commutative subalgebras in $\hat{D}$.

\begin{defin}
\label{D:rank1}
We will say that a commutative finitely generated subalgebra $B\subset \hat{D}$ is the {\it algebra of rank one} if it is $1$-quasi-elliptic strongly admissible and $\Rk (B)=1$.
\end{defin}

\begin{rm}
As we will see later in corollary \ref{C:ARR}, using the asymptotic Riemann-Roch theorem, the last property is equivalent to the following:
$$
\dim_k B_m\sim m^2/2,
$$
where $B_m=\{P\in B, \Ord(P)\le m\}$, and $\sim$ means that the function
$(\dim_k B_m) - m^2/2$ is a linear  in $m$ polynomial.
\end{rm}

Recall that commutative algebras were classified in \cite{Zheglov2013} up to the following  equivalence.

\begin{defin}
\label{D:rings}
The commutative $1$-quasi elliptic rings $B_1$, $B_2\subset \hat{D}$ are said to be equivalent if there is an invertible operator $S\in \hat{D}_1$ of the form $S=f+S^-$, where $S^-\in \hat{D}_1\partial_1$, $f\in k[[x_1,x_2]]^*$,  such that $B_1=SB_2S^{-1}$.
\end{defin}

In each equivalence class there exists a normalized ring, see \cite[Lemma 2.10]{Zheglov2013}.

\begin{defin}
\label{D:normalized}
We say that a commutative $1$-quasi-elliptic ring $B\subset \hat{D}$ is normalized if it contains a pair of operators $P,Q$ with $\ord_{\Gamma}(P)=(0,k)$, $\ord_{\Gamma}(Q)=(1,l)$ of the form
$$P=\partial_2^k+ \sum_{s=0}^{k-2}p_{s}\partial_2^{s} \mbox{\quad} Q=\partial_1\partial_2^l+ \sum_{s=0}^{l-1}q_{s}\partial_2^{s},$$ where $p_s,q_s\in \hat{D}_1$.
\end{defin}

By \cite[Rem. 2.11]{KurkeZheglov} two equivalent normalized rings differ by a special linear change of variables of the form
\begin{equation}
\label{E:special_change}
\partial_2\mapsto \partial_2+c\partial_1 +b, \mbox{\quad} \partial_1\mapsto \partial_1+d, \mbox{\quad} x_1\mapsto x_1-cx_2, \mbox{\quad} x_2\mapsto x_2
\end{equation}
with $c,b,d\in k$.\footnote{Note that linear changes $\partial_1\mapsto a\partial_1+b\partial_2$ or $x_i\mapsto x_i+b$ with $k\ni b\neq 0$ are not allowed  in the ring $\hat{D}$.}

Now let's explain what should be the corresponding spectral data of rank one.
In the case of  spectral data with {\it coherent} spectral sheaf of {\it rank one} the definition from \cite{Zheglov2013} can be simplified as follows (cf. \cite[\S 2.1]{KurkeZheglov}).
Let's introduce the following notation:
\begin{itemize}
\item
$T=\Spec k[[u,t]]\supset T_1=\Spec k[[u]]$
(defined by the equation $t=0$),
\item
$O=\Spec (k)\in T_1$,
\item
$R=k[[u,t]]$, $\cm =(u,t)\subset R$.
\end{itemize}

\begin{defin}
\label{D:altgeomdata}
A {\it coherent geometric data of rank $1$} is a triple $(X,j,\cf )$, where $X$ is an integral projective surface,
$$j:T\rightarrow X$$
is a dominant $k$-morphism and $\cf\subset j_*\co_T$ is a coherent subsheaf {\it of rank one} subject to the following conditions:
\begin{enumerate}
\item\label{GD1}
$j_*(T_1)=C\subset X$ is a curve\footnote{Notation: for a morphism of noetherian schemes $f:X\rightarrow Y$ and a closed subscheme $Z\subset X$, $f_*Z\subset Y$ is the closed subscheme defined by the ideal $\ker (\co_Y \stackrel{f^*}{\rightarrow}f_*\co_X\rightarrow f_*\co_Z)$}  (automatically integral),
and $p=j(O)$ is a point neither in the singular locus of $C$ nor of $X$.
\item\label{GD2}
$T_1\times_X\{p\}=\{O\}$, $T\times_XC=T_1$ (the fiber product is a subscheme of $T$ and $T_1$ is an effective Cartier divisor on $T$).
\item\label{GD3}
There exists an effective, very ample Cartier divisor $C'\subset X$ with cycle $Z(C')=dC$ and for all $n\ge 0$ the induced map (by the embedding $\cf \subset j_*\co_T$)
\begin{multline}
H^0(X,\cf (nC'))\rightarrow H^0(X,j_*\co_T(nC'))=H^0(T,\co_T(ndT_1))=\\
Rt^{-nd}\rightarrow Rt^{-nd}/\cm^{nd+1}t^{-nd}
\end{multline}
is an isomorphism.
\end{enumerate}
\end{defin}

\begin{rem}
\label{R:general_remark_to_geomdata}
Immediately from definition it follows that the curve $C$ does not belong to the singular locus of $X$. By \cite[Th. 3.2]{KOZ2014} the surface $X$ is Cohen-Macaulay along $C$. By \cite[Corol. 3.1]{KOZ2014} the sheaf $\cf$ is Cohen-Macaulay along $C$, hence it is locally free at all smooth points on $C$ belonging to the smooth locus of $X$.

By \cite[Theorem 3.4]{Zheglov2013} any coherent spectral data of rank one corresponds to a finitely generated commutative $k$-algebra of rank one in $\hat{D}$.

As it follows from \cite[Th.4.1]{Zheglov2013}, each commutative subalgebra in $\hat{D}$ can be extended to a Cohen-Macaulay subalgebra, i.e. we can assume additionally that $X$ is a Cohen-Macaulay (CM) surface. The spectral sheaf $\cf$ is known to be Cohen-Macaulay if $B\subset D$ \cite[Th. 3.1]{KurkeZheglov}. In general case it is not true, see example in  \cite[Rem. 6.2.1]{BurbanZheglov2017} (in terms of Schur pairs).

From the asymptotic Riemann-Roch theorem it follows that $C^2=1$, see \cite[Remark 2]{KurkeZheglov}.
\end{rem}

\begin{defin}
\label{D:geomdata_reduced}
We call $(X,C,p,\cf )$ a {\it reduced geometric data of rank $1$} if it consists of the following data:
\begin{enumerate}
\item\label{dat1}
$X$ is an integral projective surface;
\item\label{dat2}
$C$ is a reduced irreducible ample $\dq$-Cartier  divisor on $X$ and $X$ is Cohen-Macaulay along $C$;
\item\label{dat3}
$p\in C$ is a closed $k$-point, which is
regular on $C$ and on $X$;
\item\label{dat5}
$\cf$ is a torsion free coherent sheaf of rank one on $X$, which is Cohen-Macaulay along $C$ and subject to the following conditions.

Let $\hat{\co}_{X,p}\simeq k[[u,t]]$ be an isomorphism of local algebras chosen in such a way that $t$ corresponds to a formal local equation of $C$ at $p$.  Let
$\phi :\hat{\cf}_p \simeq  k[[u,t]]$ be a $\hat{\co}_{X,p}$-module isomorphism (trivialisation).
By item~\ref{dat2} there is the minimal natural number $d$ such that $C'=dC$ is a very ample Cartier divisor on $X$. Let $\gamma_n : H^0(X, \cf (nC'))\hookrightarrow {\cf}(nC')_p$ be an embedding (which is an embedding, since $\cf (nC')$ is a torsion free quasi-coherent sheaf on $X$).
Let $\epsilon_n : {\cf}(nC')_p \to \cf_p$ be the natural ${\co}_{X,p}$-module isomorphism  given by multiplication to an element $f^{nd} \in {\co}_{X,p}$, where $f \in {\co}_{X,p}$ is a local equation of $C$ at $p$. Let $\tau_n : k[[u,t]] \rightarrow k[[u,t]]/(u,t)^{nd+1}$ be the natural
ring epimorphism. We demand that the maps
$$    \tau_n \circ \phi \circ  \epsilon_n \circ \gamma_n \, : \, H^0(X, \cf (nC'))   \rightarrow   k[[u,t]]/(u,t)^{nd+1}$$
are  isomorphisms for all $n\ge 0$. (These conditions on the map $\phi$ do not depend on the choice of the trivialisation,   on the choice of the local algebras isomorphism with the given property, and on the choice of the appropriate element $f$.)
\end{enumerate}
\end{defin}

\begin{rem}
\label{R:small_rem}
If we fix an isomorphism $\hat{\co}_{X,p}\simeq k[[u,t]]$ and a trivialisation $\phi$ in definition \ref{D:geomdata_reduced}, then we obtain a geometric datum of rank one in the sense of \cite[Def. 3.10]{Zheglov2013}. This definition is equivalent to definition \ref{D:altgeomdata}, see \cite[\S 2.1.1]{KurkeZheglov}.
\end{rem}

\begin{defin}
Two reduced geometric data of rank $1$ $(X_1, C_1, p_1, \cf_1)$, $(X_2, C_2, p_2, \cf_2)$ are {\it isomorphic} if there is an isomorphism of surfaces $\beta :X_1\rightarrow X_2$ of surfaces and an isomorphism $\psi :\cf_2\rightarrow \beta_*\cf_1$ of sheaves on $X_2$ such that $\beta|_{C_1}: C_1\rightarrow C_2$ is an isomorphism of curves and $\beta (p_1)=p_2$.
\end{defin}

To formulate the main theorem of this section we need to introduce the following general form of a linear change of variables:
\begin{equation}
\label{E:generic_special_change}
\partial_2\mapsto a\partial_2+c\partial_1 +b, \mbox{\quad} \partial_1\mapsto e\partial_1+d, \mbox{\quad} x_1\mapsto e^{-1}x_1-cx_2, \mbox{\quad} x_2\mapsto a^{-1}x_2,
\end{equation}
where $a,e\in k^*$, $b,c,d\in k$.

\begin{thm}
\label{T:classification}
There is one-to-one correspondence between the set of commutative normalized finitely generated algebras of rank one up to linear changes of variables \eqref{E:generic_special_change} in $\hat{D}$ and the set of isomorphism classes of reduced geometric data of rank one.
\end{thm}

\begin{proof} The proof of this theorem essentially follows from \cite[Th. 3.4]{Zheglov2013} combined with \cite[Cor. 3.1, Prop. 3.2, Prop.3.3]{KOZ2014} and \cite[Rem. 2.11]{KurkeZheglov}. For convenience we'll give main steps of the proof here, since we'll need them in the next section.

{\bf 1.} In one direction the correspondence is as follows.
Any linear change \eqref{E:generic_special_change} is a composition of a linear change (of type) \eqref{E:special_change} and of a linear change $\partial_1\mapsto e \partial_1$, $\partial_2\mapsto a \partial_2$.
As it was already  mentioned above (\cite[Rem. 2.11]{KurkeZheglov}), any commutative normalized finitely generated algebra of rank one $B$ up to linear changes of variables \eqref{E:special_change} belongs to the same equivalence class. Then by \cite[Th. 3.2]{Zheglov2013} it uniquely corresponds to an equivalence class of a {\it Schur pair} $(A,W)$, a pair of subspaces in the ring $k[z_1^{-1}]((z_2))$ such that $A$ is an algebra isomorphic to $B$ and $W$ is a $A$ module isomorphic to $F$\footnote{The Schur pairs in dimension one were introduced in \cite{Mu1} (see also \cite{Mu2} for a review), where the classification of commutative algebras of ODO's was rewritten in their terms.  }. The proof of \cite[Th. 3.2]{Zheglov2013}  is constructive; the spaces $A$ and $W$ are obtained as follows: $A=S^{-1}BS$, $W=F\cdot S$, where $S$ is a monic operator of special type satisfying the condition $A_1$. It is defined by a pair of normalized operators from $B$ (see \cite[\S 2.3.4]{Zheglov2013}) using the analogue of Schur's theorem in dimension one (see \cite[Lemma 2.11]{Zheglov2013}). If one chooses another pair of normalized operators from $B$, then the resulting Schur pair will be equivalent to the first one. The linear change $\partial_1\mapsto e \partial_1$, $\partial_2\mapsto a \partial_2$ induces the isomorphism $z_1\mapsto e^{-1}z_1$, $z_2\mapsto a^{-1}z_2$ on the ring $k[z_1^{-1}]((z_2))$ and on the corresponding Schur pair. We note also that the Schur pair $(A,W)$ is strongly admissible in the sense of \cite[Def.3.12]{Zheglov2013}, i.e. it fulfils conditions of \cite[Th.3.3]{Zheglov2013}, since $B$  is strongly admissible.

The Schur pairs from \cite[Th. 3.2]{Zheglov2013} one to one correspond to pairs of subspaces in the space $k[[u]]((t))$ via an isomorphism
\begin{equation}
\label{psi_1}
\psi_1:k[z_1^{-1}]((z_2))\cap \Pi\simeq k[[u]]((t)) \mbox{\quad } z_2\mapsto t, z_1^{-1}\mapsto ut^{-1},
\end{equation}
where $k[z_1^{-1}]((z_2))\cap \Pi$ denotes the $k$-subspace generated by series satisfying the condition $A_1$ (see \cite[Cor.3.3]{Zheglov2013}). We will denote these pairs by the same letters $(A,W)$. Clearly, $A\cdot W\subset W$.

The space $k[[u]]((t))$ is a subspace of the two-dimensional local field $k((u))((t))$, on which the following discrete valuation of rank two $\nu \, : \, k((u))((t))^* \to \dz \oplus \dz $ is defined:
$$
\nu (f) = (m,l)  \quad \mbox{iff} \quad f= t^lu^m f_0 \mbox{, where} \quad f_0 \in k[[u]]^*+t k((u))[[t]] \mbox{.}
$$
(Here $k[[u]]^*$ means the set of invertible elements in the ring $k[[u]]$.) We also define the discrete valuation of rank one
$$
\nu_t(f)=l.
$$
The valuation $\nu_t$ induces a filtration on $A,W$, and we denote by $\tilde{A}$, $\tilde{W}$ the associated Rees algebra and Rees module correspondingly.

By \cite[Th. 3.3]{Zheglov2013} the Schur pair $(A,W)$ uniquely corresponds to algebro-geometric datum consisting of the integral projective surface $X\simeq \Proj \tilde{A}\simeq \tilde{B}$, the reduced irreducible ample $\dq$-Cartier divisor $C\simeq \Proj (gr(A))\simeq \Proj (gr(B))$, the regular point $p\in C$ which is the center of the valuation $\nu$ on $A$ (or, equivalently, the center of the valuation $\nu_B$ on $B$ defined as $\nu_B(P)=(k-\Ord (P))$, where $(k,l)=\ord_{\Gamma}(\sigma (P))$), the local $k$-algebra embedding $\pi: \hat{\co}_{X,p}\rightarrow k[[u,t]]$, the (quasi-coherent) spectral sheaf $\cf\simeq \Proj (\tilde{W})\simeq \Proj \tilde{F}$ together with an $\co_{X,p}$-module embedding $\phi :\cf_p \hookrightarrow k[[u,t]]$ such that the natural maps (defined in the same manner as in definition \ref{D:geomdata_reduced}) $H^0(X,\cf (nC')) \rightarrow k[[u,t]]/(u,t)^{ndr+1}$, where $r=rk (B)$, are isomorphisms for all $n\ge 0$. Clearly, the isomorphism $z_1\mapsto e^{-1}z_1$, $z_2\mapsto a^{-1}z_2$ composed with $\psi_1$ preserves the valuation $\nu$. Hence, the algebro-geometric datum corresponding to a Schur pair obtained after applying this isomorphism will have isomorphic surface, divisor, point and sheaf, but another embeddings $\pi$ and $\phi$ (i.e. the "local coordinates" $u,t$ will be changed). Note that the resulting {\it datum} will not be  isomorphic, in general, to the  original datum, but the isomorphisms from definition  \ref{D:geomdata_reduced}, item 4 will hold for both data. Thus, if we show that the spectral sheaf  is coherent of rank one, then the {\it reduced geometric data} will be isomorphic.

Let's show that $\cf$ is coherent and $r=rk (B)=1$. Recall that to each geometric data $(X,C,p,\cf_1,\pi ,\phi_1 )$, where $X,C,p,\pi$ are as above, and $\cf_1$ is a torsion free sheaf endowed with a $\co_{X,p}$-module embedding $(\cf_1)_p\hookrightarrow k[[u,t]]$ (e.g. for rank one Cohen-Macaulay sheaves, see \cite[Rem. 2.5]{KurkeZheglov}), one can attach a  pair of subspaces (analogue of the Schur pair)
$$W^1,A\subset k[[u]]((t)),$$
where $A$ is a filtered subalgebra of $k[[u]]((t))$ and $W^1$ a filtered module over it,
as follows (cf. \cite[\S 2.2, 2.4]{KurkeZheglov}; this construction has its origin in \cite{Parshin2001}, \cite{Pa} and was later elaborated in \cite{OsipovZheglov}, \cite{Osipov}, \cite{Zheglov2013}):

Let $f^d$ be a local generator of the ideal $\co_X(-C')_p$, where $C'=dC$ is a very ample Cartier divisor. Then $\nu (\pi (f^d))=(0,r^d)$ in the ring $k[[u,t]]$ and therefore  $\pi (f^d)^{-1}\in k[[u]]((t))$. So, we have natural embeddings for any $n >0$
$$
H^0(X, \cf_1 (nC'))\hookrightarrow {\cf_1 (nC')}_p\simeq f^{-nd} ({\cf_1}_p) \hookrightarrow k[[u]]((t)) \mbox{,}
$$
where the last embedding is the embedding $f^{-nd}{\cf_1}_p \stackrel{\phi }{\hookrightarrow } f^{-nd} k[[u,t]] {\hookrightarrow} k[[u]]((t))$. Hence we have the embedding
$$
\chi_1 \; : \; H^0(X\backslash C, \cf_1 )\simeq \limind_{n >0} H^0(X, \cf_1 (nC')) \hookrightarrow k[[u]]((t)) \mbox{.}
$$
We define $W^1 \eqdef \chi_1(H^0(X\backslash C, \cf_1))$. Analogously the embedding $H^0(X\backslash C, \co )\hookrightarrow k[[u]]((t))$ is defined (and we'll denote it also by $\chi_1$). We define $A \eqdef \chi_1(H^0(X\backslash C, \co ))$.

The filtration on $W^1,A$ is the filtration induced by the filtration $\nu_t$, namely
$$
A_n = A \cap {t}^{-nr}k[[u]][[t]], \mbox{\quad} W_n^1 = W^1 \cap {t}^{-nr}k[[u]][[t]]
$$

\begin{lem}
\label{L:coherence}
Let $B\subset \hat{D}$ be a commutative normalized finitely generated algebra of rank one, and let $F$ be its spectral module. Then the sheaf $\cf=\Proj \tilde{F}$, where $\tilde{F}=\oplus_{i=0}^{\infty}F_i\cdot s^i$, is coherent of rank one and $\rk (B)=1$.
\end{lem}

\begin{proof} Denote by $W$ the subspace  corresponding to $F$.
Let $F$ be generated by the elements $f_1,\ldots ,f_m$ as $B$-module. Denote by $f_{1,s_1}, \ldots ,f_{m,s_m}$ the corresponding homogeneous elements in $\tilde{B}$, where $s_i=\Ord (f_i)$. Consider the graded $\tilde{B}$-submodule $\tilde{F}'$ of the module $\tilde{F}$ generated by the elements $s, f_{1,s_1}, \ldots ,f_{m,s_m}$. Note that $H^0(X\backslash C,\cf )\simeq \tilde{F}_{(s)}\simeq F$. Thus, the sheaf $\cf'=\Proj \tilde{F}'$ is a coherent torsion free sheaf of rank one (since $\Rk B=1$). Consider the Cohen-Macaulaysation sheaf $\cf_1=CM(\cf')$ (see e.g. \cite[Rem.5.2]{KOZ2014}). It is also a coherent torsion free sheaf of rank one which contain $\cf'$ as a subsheaf. In particular, there is the extension of the $\co_{X,p}$-module embedding $\phi |_{\cf'}: (\cf')_p \hookrightarrow k[[u,t]]$ (induced by the embedding $\phi : (\cf )_p \hookrightarrow k[[u,t]]$) onto the module $(\cf_1)_p$. Denote by $W^1$ the subspace corresponding to $\cf_1$ with respect to this embedding. Then $W^1\supset W$ and $W^1_{nd}\supset W_{nd}$ for all $n\ge 0$.
By \cite[Rem. 2.6, 2.7]{KurkeZheglov} combined with \cite[Lem. 2.1]{KurkeZheglov} we have $H^0(X,\cf_1 (nC'))\simeq W^1_{nd}$ for all $n\ge 0$. In particular, it follows that the graded $\tilde{A}$-module $\oplus_{i=0}^{\infty} W_i$ is a submodule of the finitely generated $\tilde{A}$-module $\oplus_{i=0}^{\infty} W_i^1$. Thus, it is finitely generated and $\cf$ is a coherent sheaf of rank one.

At last, $rk (B)=1$ by \cite[Rem. 3.3]{KOZ2014}.
\end{proof}

The surface $X$ is Cohen-Macaulay along $C$ by \cite[Th. 3.2]{KOZ2014}. The sheaf $\cf$ is Cohen-Macaulay along $C$ by \cite[Corol. 3.1]{KOZ2014}. Thus, any given normalized finitely generated algebra of rank one up to linear change \eqref{E:generic_special_change} determines a reduced geometric data of rank one up to an isomorphism.

{\bf 2.} In the other direction the correspondence is more simple.  If we choose some trivialization $\phi :\hat{\cf}_p\simeq \hat{\co}_{X,p}\simeq k[[u,t]]$, where $t$ corresponds to a formal local equation of $C$ at $p$, we obtain an algebro-geometric datum  from \cite[Th. 3.3]{Zheglov2013} (we note that not every trivialisation gives a datum in the sense of \cite{Zheglov2013}, the condition on $t$ is important). By this theorem, it corresponds to a Schur pair of rank one, which corresponds to a normalized finitely generated algebra of rank one by a generalized Sato theorem \cite[Th. 3.1]{Zheglov2013}: $B:=SAS^{-1}$, where $S$ is the Sato operator from the theorem uniquely determined by the space $W$ of the Schur pair.

Another trivialisation (satisfying the condition on $t$) differs from the chosen one by an automorphism of $k[[u,t]]$ preserving the valuation $\nu$. Each such a trivialisation is a composition of an automorphism $h$ of the form
$$
h(u)=u \mbox{\quad mod \quad} (u^2)+(t), \mbox{\quad} h(t)=t \mbox{\quad mod \quad} (ut)+(t^2),
$$
and an automorphism of the form $u\mapsto c_1u$, $t\mapsto c_2t$, $c_1,c_2\in k^*$. Applying the automorphism of the first form we obtain isomorphic algebro-geometric datum (cf. \cite[Def. 2.4, Rem. 2.4]{KurkeZheglov}), which leads to the equivalent normalized finitely generated algebra of rank one, which differs from the first one by a linear change \eqref{E:special_change}.

Going back through the equivalences described above in step {\bf 1} it is easy to see that the change of local coordinates at $p$ of the form $u\mapsto c_1u$, $t\mapsto c_2t$ will lead to the Schur pair obtained by applying the same isomorphism $u\mapsto c_1u$, $t\mapsto c_2t$ to the original Schur pair. Then the resulting algebra will be the algebra obtained by applying the isomorphism of the form \eqref{E:generic_special_change}.

\end{proof}

\section{Pre-spectral data and pre-Schur pairs}

In this section we give a refinement of constructions from previous section. In particular, we show that in some cases conditions from definition \ref{D:geomdata_reduced} can be reformulated in simple algebro-geometric terms.

\begin{defin}
\label{D:pre-spectral}
We call $(X,C,\cf )$ {\it pre-spectral data of rank one} if it consists of the following data
\begin{enumerate}
\item
$X$ is a reduced irreducible projective algebraic surface over $k$.
\item
$C$ is a reduced irreducible Weil divisor not contained in the singular locus of $X$, which is also an ample $\dq$-Cartier divisor.
\item
$X$ is Cohen-Macaulay along $C$.
\item
$\cf$ is a coherent torsion free sheaf of rank one on $X$, which is Cohen-Macaulay along $C$, and such that
$$
h^0(X, \cf (nC'))=\frac{(nd+1)(nd+2)}{2}
$$
for $n\ge 0$, and $h^0(X, \cf (nC'))=0$ for $n<0$, where $C'=dC$ is an ample Cartier divisor on $X$.
\end{enumerate}
\end{defin}

\begin{rem}
\label{R:remark}
First let's note that, taking an appropriate trivialisation $\hat{\cf}_p\simeq k[[u,t]]$ of the spectral sheaf from definition \ref{D:geomdata_reduced}, the subspaces of the Schur pair will belong to $K\subset k[[u]]((t))$, where $K$ is the field of rational functions on $X$ embedded via an isomorphism from definition \ref{D:geomdata_reduced}, item 4. Namely, $W,A\subset \mathop{\limind}\limits_{n >0} f^{-n}\co_{X,p} \subset K$, and this embedding does not depend on the choice of the  local generator $f$. Analogous subspaces can be defined for any point $Q\in C$ regular on $C$ and $X$, if we choose a trivialisation $\cf_Q\simeq \co_{X,Q}$. As it was noticed in \cite[\S 2.4]{KurkeZheglov} (see also the proof of theorem \ref{T:classification}), a construction of subspaces $W,A$ with similar properties can be defined  for any coherent torsion free sheaf of rank one which is locally free at a dense open subset of $C$. In particular, such subspaces  are defined also for  pre-spectral data of rank one.

A pair of subspaces associated with a pre-spectral data and a regular on $X$ and $C$ point $Q$, as it was described above,  will be called a {\it pre-Schur pair $(A_Q, W_Q)$ associated with the pre-spectral data of rank one}.
\end{rem}

Let $(X,C,\cf )$ be  pre-spectral data of rank one. For any point $Q$ of the curve $C$, which is regular on $C$ and on $X$,  in \cite[\S 2.4]{KurkeZheglov} there were defined  torsion free sheaves $\cf_i,\cb_i$, $i\in \dz$ as $\cf_i=\Proj (\widetilde{W}_Q(i))$, $\cb_i=\Proj (\widetilde{A}_Q(i))$ 
with the properties: $\cb_i\subset \cb_{i+1}$, $\cf_i\subset \cf_{i+1}$, $\cf_0\simeq \cf$, $\cf_{id}\simeq \cf (iC')$, $\cb_{id}\simeq \co_X (iC')$ for any $i\in \dz$, $\cb_i|_C\simeq \cb_i/\cb_{i-1}$, $(\cf_i)|_C\simeq \cf_i/\cf_{i-1}$\footnote{we mean here the pull-backs of the factor-sheaves on $C$}. By \cite[Lemma 2.1, Remark 2.7]{KurkeZheglov} $(A_Q)_{nd}\simeq H^0(X,\co_X(nC'))$ and $(W_Q)_{nd}\simeq H^0(X,\cf (nC'))$; in particular, $(A_Q)_0\simeq (W_Q)_0\simeq k$. Since the subspaces $(W_Q)_i$ are defined with the help of the discrete valuation $\nu_C$, the sheaves $\cb_i, \cf_i$ defined by different points are canonically isomorphic.


\begin{thm}
\label{T:Cartier}
Let $(X,C,\cf )$ be a pre-spectral data of rank one. Then $C$ is a Cartier divisor. 
\end{thm}

\begin{proof} {\it Step 1.} First let's prove that the sheaves $\cb_i|_C$ are locally free. Choose an isomorphism $\pi :\hat{\co}_{X,Q}\simeq k[[u,t]]$ as in definition \ref{D:geomdata_reduced}, item 4). In \cite[Sec. 3.5]{KOZ2014} a construction of the generalized Schur pair was given. This construction associates to a datum $(X,C,\co_X,Q)$ a subspace $\da$ in the field $k((u))((t))$ with the following properties (see also \cite[Sec. 2.6]{KurkeZheglov}: $\da$ is a $k$-subalgebra isomorphic to $\Gamma (X\backslash C,\co_X)$; for $0\le i<d$, $n\in \dz$ the $k$-subspaces 
$$
\da (nd+i)=\frac{\da\cap t^nk((u))[[t]]}{\da\cap t^{n+1}k((u))[[t]]}
$$
are  the images of the quintets $(C,Q,\cb_i(nC')|_C, u, \phi)$ under the  Krichever map  in the $k$-subspace $\frac{t^nk((u))[[t]]}{t^{n+1}k((u))[[t]]}\simeq k((u))$, where $\phi$ are some trivialisations of the sheaves $\cb_i(nC')|_C$ at $Q$ on $C$.  In particular, as $\cb_{nd}|_C$ are invertible sheaves, the subspaces $\da (nd)$ are Fredholm, i.e. 
$$
\dim_k\da (nd)\cap k[[u]]< \infty, \mbox{\quad} \dim_k\frac{k((u))}{\da (nd)+k[[u]]}<\infty ,
$$
and $\da (0)\simeq \Gamma (C\backslash Q, \co_C)$. Since for any $k,l\in \dz$ $\da (k)\cdot \da (l)\subseteq \da (k+l)$, it follows that for any $0\le i< d$ the subspaces $\da (nd+i)$ are either zero for any $n\in \dz$ or Fredholm, i.e. $\cb_i(nC')|_C$ are torsion free of rank one. By the asymptotic Riemann-Roch theorem $\chi (X, \co_X(nC'))\sim d^2n^2/2$, where from it follows that $\da (nd+i)$ can not be zero simultaneously. Choosing an appropriate $t$, we can change the subspaces $\da (n)$. Thus, we can fix $t$ such that the subspaces $\da (n)$ will belong to $Quot (\da (0))\subset k((u))$, and the subspaces $\da (\pm d)$ will consist of rational functions on $C$ which are regular at all  singular points of $C$ (as $\cb_0(C')|_C\simeq \co_C(D)$, $D\in C_{reg}$, cf. \cite[ex.1.9, Ch. IV]{Ha}). 

Since for any invertible sheaf $\cl$ and torsion free sheaf $\cg$ of rank one on $C$ the natural map $\cg (C\backslash Q)\otimes_{\co_C(C\backslash Q)} \cl (C\backslash Q)\rightarrow (\cg\otimes_{\co_C}\cl )(C\backslash Q)$  is an isomorphism, we get $\da (nd)=\da (d)^n$ for any $n\in\dz$. In particular, $\da (nd)$ consist of rational functions on $C$, which are regular  on all singular points of $C$. Since for any $0<i<d$ $\da (nd+i)^d\subseteq \da ((nd+i)d)$, the subspaces $\da (nd+i)$ consist of rational functions regular on all singular points of $C$. Thus, the sheaves $\cb_i(nC')|_C$ are locally free for all $i,n\in \dz$. 

\smallskip

{\it Step 2.} Let's prove that for any two very ample divisors $D_1, D_2$ of degrees $\ge 3g+1$ on $C$, where $g=p_a(C)$ is the arithmetical genus of $C$, the natural map $H^0(C, \co_C(D_1))\otimes_k H^0(C, \co_C(D_2)) \rightarrow H^0(C, \co_C(D_1+D_2))$ is an  isomorphism. 

This assertion can be easily seen e.g. with the help of the Krichever map. 

First, note that any divisor of degree $n>2g+1$ is equivalent to $D_P^n:=(n-g-1)Q +P_1+\ldots +P_{g+1}$ (here $Q$ is the point from the previous step) with pairwise distinct points $P_i$, and any two such divisors are equivalent to the divisors $D_P^n=(n-g-1)Q +P_1+\ldots +P_{g+1}$, $D_Q^m:=(m-g-1)Q +Q_1+\ldots +Q_{g+1}$ with pairwise distinct points $P_i, Q_j$. 

We can choose a basis in the space $H^0(C, \co_C(D_P^n))$ that contain functions $f_0=1$ and $f_1,\ldots ,f_{g+1}$ having poles at one point $P_i$ and being regular at other $P_j$, $j\neq i$, for any $i=1,\ldots ,g+1$ (and can choose analogous basis in $H^0(C, \co_C(D_Q^m))$ with functions $h_i$). Indeed, $\dim_kH^0(C,\co_C(D_P^n))=n+1-g$ and $\dim_kH^0(C, \co_C((n-g-1)Q))= n-2g >0$, i.e. there is a $(g+1)$-dimensional vector subspace of functions, which are not regular at the points of the set $\{P_1,\ldots ,P_{g+1}\}$. 

Now let $n,m> 3g+1$. Let $W$ be the image of the quintet  $(C,Q,\co_C(D_P^n),u,\phi )$ under the Krichever map, where the trivialisation $\phi$ at $Q$ is determined by the effective Cartier divisor $D_P^n$. Let $W'$ be the image of the analogous quintet  $(C,Q,\co_C(D_Q^m),u,\phi' )$. Then (cf. \cite[Sec. 2.6.1]{KurkeZheglov}) $W_0:=k[[u]]\cap W \simeq  H^0(C, \co_C(D_P^n))$, $W_0':=k[[u]]\cap W'\simeq H^0(C, \co_C(D_Q^m))$. Now choose bases in the spaces $H^0(C, \co_C(D_P^{2g+1})), H^0(C, \co_C(D_Q^{2g+1}))$ as in the previous paragraph and note that the images of the elements $f_0,\ldots ,f_{g+1}$, $h_0,\ldots ,h_{g+1}$ in the spaces $W_0, W_0'$ have valuations $\ge (n-2g-1)$, $\ge (m-2g-1)$ correspondingly (it follows from the construction of the Krichever map). From the Riemann-Roch theorem it follows then that the spaces $W,W'$ contain  elements with any valuations $\le (n-2g-1)$, $\le (m-2g-1)$ correspondingly (and valuations of all elements $\le n$, $\le m$ correspondingly). Let's denote by $\tilde{f}_i$ some elements from $W_0$ with $\nu_Q(\tilde{f}_i)=i$, $i=0,\ldots (n-2g-1)$, and by $\tilde{h}_i$ the analogous elements from $W_0'$. 

Let $W''$ be the image of the quintet  $(C,Q,\co_C(D_P^n+D_Q^m),u,\phi'' )$ in $k((u))$.  
Note that the elements $f_0g_0$, $g_0f_i$, $f_0g_i$, $i=1,\ldots , g+1$ form a basis in the space $H^0(C, \co_C(gQ+P_1+\ldots +P_{g+1}+Q_1+\ldots +Q_{g+1}))$, and their images in $W''$ have valuations $\ge (n+m-3g-2)$. Again by the Riemann-Roch theorem it follows that $W''$ contains elements with any valuations $\le (n+m-3g-2)$. Now note that all these observations imply that 
\begin{multline}
\dim_kH^0(C, \co_C(D_P^n+D_Q^m))-\dim_k\langle H^0(C, \co_C(gQ+P_1+\ldots +P_{g+1}+Q_1+\ldots +Q_{g+1})), \\
\tilde{f}_0\cdot H^0(C, \co_C(D_Q^m)), \ldots , \tilde{f}_i\cdot H^0(C, \co_C(D_Q^m))\rangle \le n-2g-3-i
\end{multline}
for any $i=0, \ldots , (n-2g-1)$. Thus, $H^0(C, \co_C(D_P^n+D_Q^m))= H^0(C, \co_C(D_P^n))\cdot H^0(C, \co_C(D_Q^m))$.

\smallskip

{\it Step 3.} Now we can prove that $C$ is a Cartier divisor. The idea of the proof is the same as in \cite[Th. 4.1]{KurkeZheglov}, and the arguments will be very similar to the arguments from the proofs of \cite[Th. 4.1]{KurkeZheglov}, \cite[Lem. 3.3]{Zheglov2013} or \cite[Th. 2.1]{KOZ2014}. 

Recall that $X\simeq \Proj \tilde{A}$, $\tilde{A}:=\tilde{A}_Q$, and the divisor $C$ is defined by the homogeneous ideal $I=(s)$. Moreover, the graded $k$-algebra $\tilde{A}^{(d)}=\bigoplus\limits_{k=0}^{\infty} \tilde{A}_{kd}$ is finitely generated by elements from 
$\tilde{A}_1^{(d)}$ as a $k$-algebra ($C'=dC$ is a Cartier divisor). It is easy to see that for any integer $n$ there exists $m>n$ such that the graded $k$-algebra $\tilde{A}^{(md)}$ is finitely generated by elements from $\tilde{A}_1^{(md)}$. So, we can fix $n$ such that all sheaves $\cb_i|_C$, $i\ge nd$ are very ample and have degrees greater than $3g+1$. 

Let's show (recall) that for any  graded algebra $\tilde{A}^{(m)}$, finitely generated by $\tilde{A}_1^{(m)}$, the divisor $mC$ is an  effective Cartier divisor. We consider the subscheme $C'$ in $X$ which is defined by the homogeneous ideal $I^m=(s^m)$ of the ring $\tilde{A}$. The topological space of the subscheme $C'$ coincides with the topological space of the subscheme $C$ (as it can be seen on an affine covering of $X$). The local ring $\co_{X,C}$ coincides with the valuation ring of the discrete valuation corresponding to $C$ on $Quot(A)$, 
$$
\co_{X,C} = \tilde{A}_{(I)}= \{  a s^n / b s^n  \, \mid \,  n \ge 0,  a \in A_n, \,  b \in A_n \setminus A_{n-1} \} \mbox{.}
$$
The ideal $I$ induces the maximal ideal in the ring $\co_{X,C}$, and the ideal $I^m$ induces the $m$-th power of the maximal ideal. Therefore, if we will prove that the ideal $I^m$ defines an effective Cartier divisor on $X$, then the cycle map on this divisor is equal to $mC$ (see \cite[Appendix~A]{KOZ2014}). 
By~\cite[prop. 2.4.7]{EGAII} we have $X=\Proj \, \tilde{A}\simeq \Proj \, \tilde{A}^{(m)}$. Under this isomorphism the subscheme $C'$ is defined by the homogeneous ideal $I^m \cap \tilde{A}^{(m)}$ in the ring $\tilde{A}^{(m)}$. This ideal is generated by the element $s^m\in \tilde{A}_1^{(m)}$.
The open affine subsets $D_+(x_i)=\Spec \, \tilde{A}^{(m)}_{(x_i)}$ with $x_i\in \tilde{A}_1^{(m)}$ define a covering of $\Proj \, \tilde{A}^{(m)}$. In every ring $\tilde{A}^{(m)}_{(x_i)}$ the ideal $(I^{m} \cap \tilde{A}^{(m)})_{(x_i)}$ is generated by the element $s^m/x_i$. Therefore the homogeneous ideal $I^{m} \cap \tilde{A}^{(m)}$ defines an effective Cartier divisor.

Let $m>n$ be any integer such that $\tilde{A}^{(md)}$ is finitely generated by elements from $\tilde{A}_1^{(md)}$. It suffices to show that $\tilde{A}^{(md+1)}$ is also finitely generated by elements from $\tilde{A}_1^{(md+1)}$. For, in this case $mdC$ and $(md+1)C$ are Cartier divisors, whence $C$ is a Cartier divisor. The proof is by induction, using Step 2. By the properties of sheaves $\cb_i$ mentioned before this theorem, we have $H^0(C, \cb_{md+1}|_C)\simeq H^0(X, \cb_{md+1})/H^(X, \cb_{md})$, since $m$ is big enough. By the local flatness criterium (cf. \cite[Lem. 10.3A]{Ha}) the sheaf $\cb_{md+1}$ is locally free at $Q$, hence by \cite[Lem. 2.1, Cor. 2.7]{KurkeZheglov} $H^0(C, \cb_{md+1}|_C)\simeq A_{md+1}$ and thus 
$$
H^0(C, \cb_{md+1}|_C)\simeq A_{md+1}/A_{md}\simeq \da (md+1)\cap k[[u]]. 
$$
Now $\tilde{A}_2^{(md+1)}$ is generated by $\tilde{A}_1^{(md+1)}$, since $\tilde{A}_2^{(md)}\subset \tilde{A}_2^{(md+1)}$ is generated by $\tilde{A}_1^{(md)}\subset \tilde{A}_1^{(md+1)}$ and $(\da (md+1)\cap k[[u]])\cdot (\da (md+j)\cap k[[u]])=(\da (2md+j)\cap k[[u]])$ for $j=0,1$ by Step 2. By the same reason and induction $\tilde{A}_k^{(md+1)}$ is generated by $\tilde{A}_{k-1}^{(md+1)}$ and $(\da (md+1)\cap k[[u]])$. 
\end{proof}

\begin{cor}
\label{C:ARR}
Let $B\subset \hat{D}$ be a $1$-quasi-elliptic strongly admissible commutative finitely generated subalgebra.

Then the following properties are equivalent:
\begin{enumerate}
\item
$\Rk (B)=1$;
\item
$\dim_k B_m\sim m^2/2$.
\end{enumerate}
\end{cor}

\begin{proof}
Assume that $\Rk (B)=1$. Let $(A,W)$ and $(X,C,p,\cf )$ be the Schur pair and the reduced geometric data corresponding to $B$. By theorem \ref{T:classification} the rank of $(A,W)$ is one (cf. \cite[Th.3.2, 3.3]{Zheglov2013}). By theorem \ref{T:Cartier} $C$ is a Cartier divisor, and by \cite[Lem. 3.6]{Zheglov2013} we have $H^0(X,\co_X(nC))\simeq A_{n}\simeq B_{n}$. Since $C^2=1$, we have by the asymptotic Riemann-Roch theorem that  $\dim_k B_m\sim m^2/2$.

Conversely, assume $\dim_k B_m\sim m^2/2$. Then, obviously, $\rk (B)=1$. Let $(X,C,p,\cf ,\pi, \phi )$ be the geometric datum of rank $1$  constructed by \cite[Th.3.3]{Zheglov2013} and $(A,W)$ be the corresponding Schur pair. Then by definition of geometric datum and again by \cite[Lem. 3.6]{Zheglov2013} we have $h^0(X,\cf (nC'))=(nd+1)(nd+2)/2$, $h^0(X,\co_X(nC'))=\dim_k  B_{nd}\sim (nd)^2/2$. Then by the asymptotic Riemann-Roch theorem we have $C^2=1$. By \cite[Prop.3.2]{KOZ2014} the sheaf $\cf$ is coherent and $\rk (\cf )=1$.  Thus, $\Rk (B)=1$.
\end{proof}

\begin{defin}
\label{D:pre-Schur}
Let $K$ be a finitely generated field of transcendence degree two over $k$, and let $\nu_C$ be a discrete valuation on $K$.
A pair $(A,W)$ of subspaces in $K$, endowed with a decreasing valuation filtration, and satisfying the following conditions:
\begin{enumerate}
\item
$A$ is a finitely generated $k$-algebra with $\Quot (A)=K$;
\item
$W$ is a finitely generated $A$-module, where the module structure is induced by multiplication in $K$;
\item
$A_0=k$, $W_n=0$ for $n<0$ and for $n\ge 0$
$$\dim_k W_n=\frac{(n+1)(n+2)}{2}$$
\end{enumerate}
will be called a {\it pre-Schur pair of rank one}.
\end{defin}

Later (see corollary \ref{C:associated_pre-Schur-pair}) we will show that a pre-Schur pair $(A_Q, W_Q)$ associated with the pre-spectral data is a pre-Schur pair.

A pre-spectral data of rank one can be completed to a coherent geometric data of rank one. To prove this statement, we need to recall several facts and constructions from \cite{KurkeZheglov}.

The following property was discussed in \cite[\S 2.6.1]{KurkeZheglov}.
\begin{property}
\label{property}
If $\cg$ is a torsion free sheaf on the curve $C$ with $h^0(C,\cg )=l$, then there exists a dense open subset $U\subset C$ such that for every point $Q\in U$ the function
$$
f_Q(m)=h^0(C,\cg (-mQ))
$$
is strictly monotonic for $0\le m\le l$. In particular, $H^0(C, \cg (-lQ))=0$.
\end{property}
This property easily follows from the following observation: for any fixed section $a\in H^0(C, \cg)$ there is a dense open subset $U\subset C$ such that for any $Q\in U$ the image of $a$ in $\co_{C,Q}$ with respect to any trivialisation is invertible.

\begin{prop}
\label{P:extension_pre-spectral}
Let $(X,C,\cf )$ be a  pre-spectral data of rank one.

Then this data can be extended to a coherent geometric data of rank one, in particular, to a reduced geometric data of rank one.
\end{prop}

\begin{proof}
We need to show that there exists a smooth point $p\in C$ such that the conditions of definition \ref{D:geomdata_reduced}, item \ref{dat5} are satisfied (see remark \ref{R:small_rem}).

Using the exact sequence
$$
0 \rightarrow H^0(X, \cf ((i-1)C)) \rightarrow H^0(X, \cf (iC)) \rightarrow H^0(C, \cf (iC)|_C),
$$
the conditions of definition \ref{D:geomdata_reduced}, item \ref{dat5} can be reformulated as follows: the function
$$
\tilde{f}_p(m)=\dim_k\left(\frac{H^0(X, \cf (iC))}{H^0(X, \cf ((i-1)C))}\cap H^0(C, (\cf (iC)|_C)(-mp))\right)
$$
is strictly monotonic for $0\le m\le i$ (the intersection is taken in $H^0(C, \cf (iC)|_C)$).

By the same observation above that explains property \ref{property}, for any fixed $i\ge 0$ there is a dense open subset $U_{i}\subset C$ such that for any $p\in U_{i}$ the function $\tilde{f}_p(m)$ is strictly monotonic for $0\le m\le i$. Therefore, since the ground field $k$ is uncountable, there exists a point $p\in \cap_{i\ge 0}U_{i}$ regular in $C$ and $X$ such that these properties hold simultaneously. Taking some formal local parameters $u,t$ at such $p$ (as in \ref{D:geomdata_reduced}, item \ref{dat5}), and choosing any trivialisation of $\cf$ at $P$ we obtain a coherent geometric  data of rank one.
\end{proof}

\begin{cor}
\label{C:associated_pre-Schur-pair}
Let $(A_Q, W_Q)$ be the pre-Schur pair associated with the pre-spectral data $(X,C,\cf )$ of rank one. Then $(A_Q, W_Q)$ is a pre-Schur pair of rank one.
\end{cor}

\begin{proof} Since $C$ is a Cartier divisor, then $(W_Q)_n=(A_Q)_n=0$ for $n< 0$. For, for $i<0$  $(A_Q)_0\supseteq (A_Q)_i$,  $k\nsubseteq (A_Q)_i$ as the valuation of any constant is zero, and $(W_P)_i\simeq H^0(X,\cf (iC))=0$ by definition.  So, 
we need to check only that for $n\ge 0$ $\dim_k (W_Q)_n=(n+1)(n+2)/2$. As we have already noticed, $(W_Q)_{i}\simeq H^0(X, \cf (iC))$. If $(X,C,p,\cf )$ is a reduced geometric data of rank one, an extension of $(X,C,\cf )$, then there are isomorphisms $H^0(X,\cf (nC))\simeq k[[u,t]]/(u,t)^{n+1}$. As $(W_p)_i$ are defined with the help of the discrete valuation $\nu_C =\nu_t|_K$, we get $(W_p)_{i}\simeq k[[u,t]]/(u,t)^{i+1}$. Thus, $\dim_k H^0(X, \cf (nC))= (n+1)(n+2)/2$ and we are done.
\end{proof}

\begin{prop}
\label{P:pre-spectral_reconstruction}
Let $(A,W)$ be a pre-Schur pair of rank one.

Then this pair defines a  pre-spectral data of rank one, where
$X=\Proj \tilde{A}$, $\tilde{A}=\bigoplus\limits_{n=0}^{\infty}A_n s^n$, $C=\Proj (gr(\tilde{A}))$, $\cf =\Proj \tilde{W}$, $\tilde{W}=\bigoplus\limits_{n=0}^{\infty}W_n s^n$.
\end{prop}

\begin{proof}
First let's show that $\tilde{A}$ is finitely generated.
Let $A$ be generated by  elements $t_1,\ldots ,t_m$ as  $k$-algebra. Denote by $t_{1,s_1},\ldots ,t_{m,s_m}$ the corresponding homogeneous elements in $\tilde{A}$, where for each $i$ we denote by $s_i$ the minimal number such that $t_i\in A_{s_i}$. Consider the finitely generated $k$-subalgebra $\tilde{A}_1=k[s,t_{1,s_1},\ldots , t_{m,s_m}]\subset \tilde{A}$. By
\cite[Lemma 3.3]{Zheglov2013}, item 1) (see also the proof of \cite[Theorem 2.1]{KOZ2014}), it defines an irreducible projective surface $X=\Proj (\tilde{A})$, with an irreducible ample $\dq$-Cartier divisor $C=\Proj (gr(\tilde{A}))$ corresponding to the valuation $\nu_C$, which does not belong to the singular locus of $X$. Besides, $H^0(X\backslash C, \co_X)\simeq (\tilde{A}_1)_{(s)}\simeq A$.
If $d\ge 1$ is a minimal integer such that the $k$-algebra $\tilde{A}_1^{(d)}=\bigoplus\limits_{l=0}^{\infty}(\tilde{A}_1)_{ld}$ is finitely generated by elements from $(\tilde{A}_1)^{(d)}_1$ as a $k$-algebra (such $d$ exists by \cite[Ch.III, \S 1.3, prop.3]{Bu}), then $dC$ is a very ample Cartier divisor, and by \cite[Lemma 2.1, Remark 2.7]{KurkeZheglov} (applied to an arbitrary smooth on $X$ and $C$ point $Q$) $H^0(X, \co_X(nC'))\simeq A_{nd}$ (cf. \cite[Lem.3.6]{Zheglov2013}). Thus, the ring $\bigoplus\limits_{n=0}^{\infty}A_{nd}$ is finitely generated, hence the ring $\tilde{A}$ is finitely generated and $X\simeq \Proj \tilde{A}$.

Now just the same arguments as in the proof of \cite[Th. 3.2]{KOZ2014} (namely, the arguments starting from the 4-th sentence of the proof) can be applied to show that $X$ is Cohen-Macaulay along $C$.

The arguments analogous to the proof of \cite[Lemma 3.8]{Zheglov2013} or to the proof of lemma \ref{L:coherence} can be applied to show that the sheaf $\cf =\Proj \tilde{W}$ is coherent. Namely, we can consider a finitely generated graded $\tilde{A}$-module $\tilde{W}_1$, which contains all generators of $W$ over $A$. Then $\cf'=\Proj \tilde{W}_1$ is a coherent torsion free sheaf of rank one, since $(\tilde{W}_1)_{(s)}=W\subset K$. By \cite[Corol. 3.1]{KOZ2014} this sheaf is Cohen-Macaulay along $C$ (we refer to Corol. 3.1 since its proof is valid also in our situation). Then by \cite[Lemma 2.1, Remark 2.7]{KurkeZheglov} $H^0(X, \cf'(nC'))\simeq W_{nd}$ for all $n\ge 0$, thus $\cf'\simeq \Proj (\oplus_{n=0}^{\infty}W_{nd})\simeq \cf$ by \cite[Prop. 2.4.7]{EGAII}.
\end{proof}

\section{Examples}

The results of previous section show that in order to find examples of rank one commutative subalgebras in $\hat{D}$, it is enough to find examples of pre-spectral data. Conjecturally, the operators of such algebras could be expressed explicitly, in analogy with one-dimensional situation of ordinary differential operators or difference operators, see \cite{Krichever77}, \cite{Krichever78}, \cite{KricheverNovikov2003}.

It is easy to see from the results of previous section that "trivial" algebras (see Introduction) correspond to spectral data with the property $h^0(X,\co_X(C))\ge 2$ (cf. \cite[Th.4.1]{KurkeZheglov}). Indeed, as we have recall in the course of proof of theorem \ref{T:classification}, $H^0(X,\co_X(C))\simeq A_2\simeq B_2$. Thus the spectral surfaces of such algebras endow a pencil of curves (some examples,  with singular spectral surfaces, see in \cite[\S 4]{KurkeZheglov}).

In this section we investigate the question whether there are pre-spectral data with a smooth spectral surface and $h^0(X,\co_X(C))=1$. We construct first examples of such data with minimal possible genus of the divisor $C$.

\begin{prop}
\label{P:genus_of_C}
Let $(X,C,\cf )$ be a pre-spectral data of rank one with a smooth surface $X$ and such that the arithmetical genus $p_a(C)\le 1$. Then $h^0(X,\co_X(C))\ge 2$.
\end{prop}

\begin{proof}
If $p_a(C)=0$, then $C\simeq \dpp^1$, $X\simeq \dpp^2$ and $h^0(X,\co_X(C))= 3$, see e.g. the end of the proof of \cite[Th.4.1]{KurkeZheglov}.

If $p_a(C)=1$ and $C$ is singular, then $C$ is rational. Then the Albanese morphism $f: X \rightarrow Alb(X)$ must have trivial image, since the ample and rational curve C can not be mapped to a rational curve of $Alb(X)$. This implies the vanishing of $H^1(X, \co_X)$ and the exact sequence
\begin{equation}
\label{E:sequence}
0\rightarrow H^0(X,\co_X)\rightarrow H^0(X,\co_X(C))\rightarrow H^0(C,\co_C(C))\rightarrow H^1(X,\co_X)=0
\end{equation}
yields $h^0(X,\co_X(C))= 2$.

If $p_a(C)=1$ and $C$ is smooth, then we can apply the arguments from \cite[Th.2.5.19]{Badescu}. Namely, since $p_a(C)=1$ and $C^2=1$, we must have $h^1(C,\co_C(C))=0$, hence the truncated exponential sequence
$$
0\rightarrow \co_C(-C)\rightarrow \co_{C(1)}^*\rightarrow \co_C^*\rightarrow 0,
$$
where $C(1)$ means the first infinitesimal neighbourhood of $C$ in $X$, 
yields the isomorphism of algebraic groups $\underline{\Pic}^0(C(1))\simeq \underline{\Pic}^0(C)$, because the linear algebraic group $\underline{H^1(C,\co_C(-C))}$ (regarded as a product of $h^1(C,\co_C(C))$ copies of the additive group $\dg_a$) is trivial. 

Now we have two possibilities. If the canonical restriction map $\alpha :\Pic (X)\rightarrow \Pic (C)$ is surjective, then 
the map $\Pic(X)\rightarrow \Pic (C(1))$ is surjective and then 
we are in the situation of theorem 2.5.19 from \cite{Badescu}, and $C$ must be isomorphic to $\dpp^1$, a contradiction. If $\alpha$ is not surjective, it induces a trivial map of algebraic groups $\underline{\Pic}^0(X)\rightarrow \underline{\Pic}^0(C)$. Then $H^1(X,\co_X)=0$ and again by \eqref{E:sequence} $h^0(X,\co_X(C))= 2$.
\end{proof}

The following definition is motivated by the fact that Cohen-Macaulay sheaves form an open subset in the moduli space of rank one torsion free sheaves with fixed Hilbert polynomial (see \cite[Th. 12.2.1]{EGAIV}). 

\begin{defin}
\label{D:excellent}
Let $X,C$ be a surface and divisor satisfying conditions from definition \ref{D:pre-spectral}.  We will say that a rank one torsion free sheaf $\cf$ is {\it excellent} if $H^1(C,\cf|_C)=0$ and
$$
\chi (\cf (nC'))=\frac{(nd+1)(nd+2)}{2}.
$$
\end{defin}

If $X$ is smooth, then $C$ is a Cartier divisor, hence $d=1$. Then by property \ref{property} there is a point $Q$ such that the conditions from \cite[Prop. 2.3]{KurkeZheglov} hold. Then $H^1(X,\cf )=H^2(X,\cf )=0$ and $\cf$ is Cohen-Macaulay, hence it is locally free.

\begin{prop}
\label{P:excellent}
Let $X,C$ be a smooth surface and divisor satisfying conditions from definition \ref{D:pre-spectral}, let $\cf$ be a torsion free sheaf of rank one. Then the following conditions are equivalent:
\begin{enumerate}
\item
$\cf$ is an excellent sheaf;
\item
$H^i(X,\cf (-C))=0$, $i=0,1,2$; $h^0(X,\cf )=1$ and
$$
\chi (\cf ((n-1)C))=\frac{n(n+1)}{2}.
$$
\end{enumerate}
\end{prop}

\begin{proof}
If $\cf$ is an excellent sheaf, then, obviously,
$$
\chi (\cf ((n-1)C))=\frac{n(n+1)}{2}.
$$
Now from the exact sequence
\begin{equation}
\label{E:trivial}
0\rightarrow \cf (-C)\rightarrow \cf \rightarrow \cf|_C\rightarrow 0
\end{equation}
and from the Riemann-Roch theorem it follows $H^i(X,\cf (-C))=0$, $i=0,1,2$. At last, as we have already seen above,  by \cite[Prop. 2.3]{KurkeZheglov} we have $H^1(X,\cf )=H^2(X,\cf )=0$ and $h^0(X,\cf )=1$.

If $\cf$ satisfies the conditions of item 2), then, clearly,
$$
\chi (\cf (nC))=\frac{(n+1)(n+2)}{2},
$$
and again from the sequence  \eqref{E:trivial}  it follows immediately that $H^1(C,\cf|_C)=0$, i.e. $\cf$ is excellent.
\end{proof}

\begin{thm}\label{T:example} There is an eight-dimensional
family 
of pairwise non-isomorphic Godeaux surfaces $X$ such that on each $X$
from this family there  are 1200 different divisors $D_j$ and four curves $C_i$ satisfying the following conditions:
\begin{itemize}
\item[$(i)$] $C_i$  is  smooth, the sheaf $\mathcal O_X(C_i)$ is ample, and 
$\dim H^0(X,\mathcal O_X(C_i))=1$;
\item[$(ii)$]
$C_i^2=1$, $g(C_i)=2$, where $g(C_i)$ is the genus of $C_i$;;
\item[$(iii)$] $(D_j,C_i)_X=g(C_i)-1$;
\item[$(iv)$] 
$\chi(D_j)=\sum_{i=0}^2(-1)^i\dim H^i(X,\mathcal O_X(D_j))=0$.
\end{itemize}
Moreover, among the divisors $D_j$ there are at least 1080 different divisors such that \\
$\dim H^i(X,\mathcal O_X(D_j))=0$ for $i=0,1,2$, and among them at least 840 divisors such that $\dim H^0(X,\co_X(D_j+C_i))=1$.
The triples $(X,C_i,D_j+C_i)$, where $X$ is any surface from this family, and $D_j$ is any representative of the last class of divisors, are pre-spectral data of rank one.
\end{thm}


\proof
Denote by $$\begin{array} {llll} N=\{ &  \overline n_1=(5,0,0,0),\, \, \overline n_2=(3,0,1,1), & \overline n_3= (2,1,2,0), & \overline n_4=(2,2,0,1), \\ & \overline n_5=(1,3,1,0),\ \, \overline n_6=(1,1,0,3), & \overline n_7= (1,0,2,2), & \overline n_8= (0,5,0,0), \\
 & \overline n_9=(0,0,5,0),\, \, \overline n_{10}=(0,0,0,5), & \overline n_{11}=(0,2,1,2), &
 \overline n_{12}=(0,1,3,1)\} \end{array}$$
the set of non-negative integer solutions of the equations
\begin{equation} \label{equ} \sum_{i=1}^4n_i=5; \qquad \qquad \sum_{i=1}^4in_i\equiv 0\, \text{mod}\, 5
\end{equation}
and consider the family $\overline{\mathcal Q} \subset \mathbb P^3\times \mathbb P^{11}$  of quintics $Q_{\overline a}\subset \mathbb P^3$ given by equation
$$ \sum_{\overline n_i\in N} a_{i}\overline z^{\overline n_i}=0,$$
where $\overline z=(z_1:z_2:z_3:z_4)\in \mathbb P^3$, $\overline z^{\overline n_i}=z_1^{n_1}z_2^{n_2}z_3^{n_3}z_4^{n_4}$ for $\overline n_i=(n_1,n_2,n_3,n_4)$ and $\overline a=(a_1:\dots :a_{12})\in \mathbb P^{11}$.

Denote by $\overline G$  the subgroup of $PGL(4,\mathbb C)$ generated by elements $g_j\in PGL(4,\mathbb C)$, $j=1,2,3$, acting on $\mathbb P^3$ as follows: $g_j(z_j)=\varepsilon z_j$ and $g_j(z_k)=z_k$ for $k\neq j$, where $\varepsilon= e^{2\pi i/5}$.
It is easy to see that the cyclic subgroup
$G\simeq \mathbb Z_5$ of the group $\overline G$ acting on the projective space $\mathbb P^3$ as follows:
$$(z_1:z_2:z_3:z_4)\mapsto (\varepsilon z_1:\varepsilon^2 z_2:\varepsilon^3z_3:\varepsilon^4z_4)$$
 leaves invariant each $Q_{\overline a}\subset \overline{\mathcal Q}$.

Let $\mathcal Q$ be the subfamily of the family $\overline{\mathcal Q}$ consisting of smooth quintics such that $G$ acts on each $Q_{\overline a}\subset \mathcal Q$ freely.
Then  the factor spaces $X_{\overline a}=Q_{\overline a}/G$, $Q_{\overline a}\subset \mathcal Q$, are smooth surfaces of general type (the Godeaux surfaces; see, for example, \cite{B-H-P-V}) which have the following invariants:
$p_g=q=0$, $K_{X_{\overline a}}^2=1$, and $Tor(X_{\overline a}):=Tor(Pic(X_{\overline a}))\simeq\mathbb Z_5$.

Denote by $\mathcal Q_0$ the subfamily of $\mathcal Q$ consisting of the quintics $Q_{\overline a}$ such that the coordinate planes $\{ z_i=0\}$, $i=1,\dots, 4$, meet transversally $Q_{\overline a}$, and let $A\subset \mathbb P^{11}$ (resp., $A_0\subset \mathbb P^{11}$) be the image of $\mathcal Q$ (resp., $\mathcal Q_0$) under the projection of $\mathbb P^3\times \mathbb P^{11}$ to the second factor $\mathbb P^{11}$. Note that $A$ and $A_0$ are non-empty  Zariski open subsets of $\mathbb P^{11}$, since Fermat quintic (denote it by $Q_F$) is a member of the family $\mathcal Q_0$.
Fermat quintic $Q_F$ is given by $\sum_{j=1}^4 z_j^5=0$. The group $\overline G$ leaves invariant the surface $Q_F$. In \cite{K-O}, it was shown that the factor-space $Q_F/\overline G =\mathbb P^2$ and the factorization map $\zeta: Q_F\to \mathbb P^2$ is branched in four lines $L_1,\dots, L_4$ in general position with multiplicity five. Let  $X_F=Q_F/G$ be the corresponding Godaux surface and $\varphi :Q_F\rightarrow X_F$ be the factor map. The map $\zeta$ factorizes through $\varphi$, i.e., there is a map $\psi: X_F\to \mathbb P^2$ such that $\zeta=\psi\circ\varphi$. The map $\psi$ is a Galois covering with Galois group $H=\overline G/G\simeq \mathbb Z_5^2$. It  is branched in lines $L_1,\dots, L_4$ with multiplicity five and $\psi^*(L_j)=5C_j$ for $j=1,\dots, 4$ and it is given by the epimorphism $\psi^*(\pi_1(\mathbb P^2\setminus (\cup L_j))\to H$
 Note that it was shown in \cite{K-O} that the Godaux surface $X_F$ is the normalization of the surface in $\mathbb P^4$ given by
$$\begin{array}{ccl}
z_4^5 & = & z_1z_2^2(z_1+z_2+z_3)^2, \\
z_5^5 & = & z_2z_3(z_1+z_2+z_3)^3.
\end{array}
$$

\begin{claim} \label{cl1}  For any $\overline a_1$, $\overline a_2\in A$, the surfaces $X_{\overline a_1}$ and $X_{\overline a_2}$ are isomorphic if and only if there is
$$ \Lambda=\left( \begin{array}{cccc} \lambda_1, & 0, & 0, & 0 \\
0,& \lambda_2, & 0,& 0 \\
0, & 0, & \lambda_3, & 0 \\
0, & 0, & 0, & \lambda_4 \end{array} \right)
\in PGL(4,\mathbb C) $$
such that $\Lambda(Q_{\overline a_1})=Q_{\overline a_2}$.
 \end{claim}

\proof The surfaces  $X_{\overline a_1}$ and $X_{\overline a_2}$ are isomorphic if and only if their universal coverings $Q_{\overline a_1}$ and $Q_{\overline a_2}$ are isomorphic as $G$-manifolds. An isomorphism between quintics $Q_{\overline a_1}$ and $Q_{\overline a_2}$ can be given by linear transformation
$\Lambda$ of $\mathbb P^3$, since the canonical class of a smooth quintic is its hyperplane section. Therefore, if $\Lambda$ preserves the structure of $G$-manifolds on $Q_{\overline a_1}$ and $Q_{\overline a_2}$, it must commute with projective transformation
$$\Lambda_5 = \left( \begin{array}{cccc} \varepsilon, & 0, & 0, & 0 \\
0,& \varepsilon^2, & 0,& 0 \\
0, & 0, & \varepsilon^3, & 0 \\
0, & 0, & 0, & \varepsilon^4 \end{array} \right) .$$ \qed

\begin{claim} \label{cl2} Let $X=X_{\overline a_0}$, $\overline a_0\in A_0$, be one of the Godeaux surfaces. There are four curves $C_j\subset X$, $j=1,\dots, 4$, numerically equivalent to $K_X$ and satisfying condition $(i)$ from Theorem \ref{T:example}.

For each point $p\in X$ there are not three curves from the set $\{ C_1,\dots, C_4\}$ passing through $p$.
\end{claim}
\proof We have $X=Q/G$, where $Q=Q_{\overline a_0}$. Denote by $\varphi: Q\to X$ the factor map defined by the action of $G$.

Let $\alpha_1$ be a generator of the group $Tor(X)\simeq \mathbb Z_5$. By Serre's duality,  for $j=1,\dots, 4$ we have $H^2(X,\mathcal O_X(K_X+j\alpha_1))=0$. Therefore, it follows from Riemann-Roch Theorem  that $\dim H^0(X,\mathcal O_X(K_X+j\alpha_1))>0$.

Let
$C_j\in |K_X+j\alpha_1 |$, $j=1,\dots,4$. The curves $\varphi^{-1}(C_j)$ belong to the canonical class $K_Q$ of $Q$. Therefore $\varphi^{-1}(C_j)$ are hyperplane sections of $Q$ invariant under the action of $G\simeq \mathbb Z_5$. But, it is easy to see that there are only four planes, namely, $\{ z_i=0\}$, invariant under the action of $\mathbb Z_5$ on $\mathbb P^3$.
Therefore the curves $C_j$, $j=1,\dots, 4$, are the images of smooth curves (the intersections of planes $\{ z_i=0\}$ with $Q$), since $\overline a_0\in A_0$.
Hence, $\dim H^0(X,\mathcal O_X(K_X+j\alpha_1))=1$ and the curves $C_j$ are smooth of genus  $g(C_j)=2$.

The quintic $Q$ is  smooth and its imbedding in $\mathbb P^3$ is given by its canonical class. Therefore there are not $(-2)$-curves lying in $Q$. It follows from this that $X$ can not have $(-2)$-curves, since $\varphi$ is unramified covering and $(-2)$-curves are simply connected. Therefore  the divisors $C_j$ are ample, since they are numerically equivalent to the canonical class of $X$.

For each point $p\in X$ there are not three curves from the set $\{ C_1,\dots, C_4\}$ passing through $p$, since the curves $C_j$ are the images of coordinate plane sections of $Q$ under the map $\varphi$.\qed \\

\begin{claim}\label{cl5} The morphism $i_{j*}:Tor(X)\to Pic(C_j)$, induced by imbedding $i_j: C_j\hookrightarrow X$, is a monomorphism.
\end{claim}
\proof Denote by $p_{j,l}=C_j\cap C_l$. Then
$$p_{j,j+2}-p_{j,j+1}=i_{j*}(C_{j+2})-i_{j*}(C_{j+1})= i_{j*}(K_X+(j+2)\alpha_1)-i{_*}(K_X+(j+1)\alpha_1)=\alpha_1$$
(here $j+l$ in the indexes  of $p_{j,j+l}$ and $C_{j+l}$ are considered modulo 5). By Claim \ref{cl2}, the divisor $p_{j,j+2}-p_{j,j+1}\neq 0$ in $Pic(C_j)$, since $g(C_j)=2$. \qed

\begin{claim} \label{cl3} There are 1200 different elements $E_j\in Pic(X)$ such that $(E_j^2)_X=-2$ and $(E_j,K_X)_X=0$.
\end{claim}
\proof Since $\chi(X)=1-q+p_g=1$ and $K_X^2=1$, then, by Noether's formula, the topological Euler characteristic $e(X)=11$ and, consequently,
$\dim H^2(X,\mathbb C)=9$. We have  $H^2(X,\mathbb C)=H^{1,1}(X,\mathbb C)$, since $p_g=0$ and therefore
$Pic(X)\simeq H^2(X,\mathbb Z)$. Pick
$P=Pic(X)/Tor(X)$.
By Hodge index theorem, $P$ is an unimodular odd lattice of signature $\sigma=(1,8)$. It follows from equalities $K_X^2=1$ and  $(K_X,D)_X\equiv (D^2)\, mod\, 2$ for each $D\in Pic(X)$ that
$P=\mathbb ZK_X\bigoplus K_X^{\bot }$, where $K_X^{\bot}$ is an even negatively defined unimodular lattice of rank 8. Therefore $K_X^{\bot}$ is isomorphic to the lattice $-E_8$ (see \cite{Serre} or \cite{Bur}) and hence there exist 240 elements $e_l\in P$, $1\leq l\leq 240$, such that $(e_l^2)=-2$ and $(K_X,e_l)=0$. The preimages $E_j$, $1\leq j\leq 1200$, in $Pic(X)$ of the elements $e_l$ are the desired elements. \qed\\

Denote by $\mathcal D$ the set of divisors $D_j=K_X+E_j$, where $E_j$ are divisors from Claim \ref{cl3}. The set $\mathcal D$ can be divided into disjoint union $\mathcal D=\bigsqcup_{l=1}^{120}\mathcal D_l$ of 120 subsets, where if $D_{j}=K_X+E_j\in \mathcal D_l$, then $K_X+E_j+\alpha\in \mathcal D_l$ and $K_X-E_j+\alpha\in\mathcal D_l$ for all $\alpha\in Tor(X)$. We say that $D_j\in \mathcal D_l$ is a {\it bad} element of $\mathcal D_l$ if $\dim H^0(X,\mathcal O_X(D_j))\geq 1$.

\begin{claim} \label{cl4}
 The curves $C_i$ from Claim \ref{cl2} and the divisors $D_j\in \mathcal D$ satisfy conditions $(ii)$ and $(iv)$ from Theorem \ref{T:example}. Moreover, each set $\mathcal D_l$ can contain at most one bad element, that is, among the divisors $D_j\in\mathcal D$ there are at least 1080 different divisors such that $\dim H^i(X,\mathcal O_X(D_j))=0$ for $i=0,1,2$.
\end{claim}
\proof We have $(D_{j}^2)_X=-1$ for $1\leq j\leq 1200$ and $(D_{j},K_X)_X=(D_{j},C_i)_X=1=g(C_i)-1$ for $1\leq i\leq 4$, and by Riemann-Roch Theorem, $\chi(\mathcal O_X(D_j))=0$.

Note that $\dim H^0(X,\mathcal O_X(E_j))=0$ for  $1\leq j\leq 1200$, since the canonical class of $X$ is ample. Therefore, by Serre's duality, $H^2(X,\mathcal O_X(D_j))=0$. Also, note that if $\dim H^0(X,\mathcal O_X(K_X+E_j))\geq 1$, then $M\in |K_X+E_j|$ is an irreducible reduced curve, since, first, $(M,K_X)_X=1$ and, second, $K_X$ is an ample divisor and therefore $(M,K_X)_X>0$ for each curve $M\subset X$.
Therefore to complete the proof of Claim \ref{cl4}, 
it is sufficient to show that if 
$\dim H^0(X,\mathcal O_X(K_X+E_j))\geq 1$, then  $\dim H^0(X,\mathcal O_X(K_X-E_j+\alpha))=0$ for $\alpha\in Tor(X)$ and
$\dim H^0(X,\mathcal O_X(K_X+E_j+\alpha))=0$ for $\alpha\in Tor(X)$,   $\alpha\neq 0$.

Assume that $\dim H^0(X,\mathcal O_X(K_X+E_j))\geq 1$ and  $\dim H^0(X,\mathcal O_X(K_X+E_j+\alpha))\geq 1$ for some
$E_j$ and $\alpha\in Tor(X)$, $\alpha\neq 0$, and let $M_1\in |K_X+E_j|$ and $M_2\in |K_X+E_j+\alpha|$ be two curves in $X$.
Then $M_1$ and $M_2$ are two different irreducible curves such that $(M_1,M_2)_X=(M_1^2)_X=-1$. Contradiction.

Assume that $\dim H^0(X,\mathcal O_X(K_X+E_j))\geq 1$ and  $\dim H^0(X,\mathcal O_X(K_X-E_j+\alpha))\geq 1$ for some
$E_j$, and let $M_1\in |K_X+E_j|$ and $M_2\in |K_X-E_j+\alpha|$ be two curves in $X$. Then $M_1$ and $M_2$ are two different irreducible curves such that $(M_1,M_2)_X=3$.  Consider the curve $\widetilde D=\varphi^{-1}(M_1\cup M_2)\subset Q$, where $Q$ is a smooth quintic in $\mathbb P^3$ and $\varphi:Q\to X$ is the universal covering. Then $\widetilde D\in |2K_Q|$, since $M_1\cup M_2\in |2K_X+\alpha|$ and therefore $\widetilde D=Q\cap S$, where $S$ is  a quadric in $\mathbb P^3$ invariant under the action of $G$. Note that if $S$ is a cone then $Sing S\not\in Q\cap S$, since $Sing S$ is a single point and the action of $G$ on $Q$ is free from fixed points. Note also that $S$ is not the union of two planes $P_1$ and $P_2$, since overwise this planes must be invariant under the action of $G$ and hence $P_i\cap Q=\varphi^{-1}(C_{j(i)})$, where $C_{j(i)}$ is a curve numerically equivalent to $K_X$.

Denote $\widetilde D_i=\varphi^{-1}(M_i)$, $i=1,2$. Then
\begin{equation}\label{e1}(\widetilde D_1,\widetilde D_2)_Q=5(M_1,M_2)_X=15. \end{equation}
Let us show that $(\widetilde D_1,\widetilde D_2)_Q= (\widetilde D_1,\widetilde D_2)_S$. For this, consider any point $p\in \widetilde D_1\cap\widetilde D_2$. The surfaces $Q$ and $S$ have a common tangent plane $P$ at $p$, since $Q\cap S$ is singular at $p$. Let us choose non-homogeneous coordinates $(x,y,z)$ in $\mathbb A^3\subset \mathbb P^3$ such that $p=(0,0,0)$ and $z=0$ is an equation of $P$, and consider an irreducible component $J$ of the germ $(\widetilde D_1,p)$ (resp., $\widetilde D_2$) of $\widetilde D_1$ (resp., $\widetilde D_2$) at $p$. The germ $J$ can be given parametrically by $x=h_1(t)$, $y=h_2(t)$, and $z=h_3(t)$, where $h_i(t)\in \mathbb C[[t]]$. The functions $x$ and $y$ are local coordinates at $p$ in both $Q$ and $S$, since $\{ z=0\}$ is the tangent plane to $Q$ and $S$ at $p$. Therefore $J$ is given in $Q$ and in $S$ by the same parametrisation $x=h_1(t)$, $y_1=h_2(t)$ and hence the intersection number $(\widetilde D_1,\widetilde D_2)_p$ of the curves $\widetilde D_1$ and $\widetilde D_2$ at $p$ is the same in the cases if we consider the curves $\widetilde D_1$ and $\widetilde D_2$ as the curves lying in $Q$ or lying in $S$.

Let $S\simeq \mathbb P^1\times \mathbb P^1$ be a smooth quadric. The group $Pic(S)=\mathbb Z^2$ is generated by $L_1$ and $L_2$, $(L_1^2)_S=(L_2^2)_S=0$, and $(L_1,L_2)_S=1$, where $L_1$ and $L_2$ are fibres of two projections of $S$ to $\mathbb P^1$. We have $\widetilde D\in |5L_1+5L_2|$. Let $\widetilde D_1\in |mL_1+nL_2|$, $m,n\geq 0$. Then $\widetilde D_2\in |(5-m)L_1+(5-n)L_2|$, $m,n\leq 5$. Therefore
$(\widetilde D_1,\widetilde D_2)_S= 5(m+n)-2mn$ and, by (\ref{e1}), we have
$$5(m+n)-2mn=15.$$
Therefore $mn$ is divisible by $5$ and hence $(m,n)\in \{ (0,3),(3,0), (5,2), (2,5)\}$, and, without loss of generality, we can assume that $(m,n)=(0,3)$, that is, the curve $\widetilde D_1$ is equal to the union $L_{2,1}\cup L_{2,2}\cup L_{2,3}$, where $L_{2,j}\in |L_{2}|$, $j=1,2,3$, are three pairwise different curves, since $\widetilde D_1$ is a reduced curve. But, the curve $\widetilde D_1$ is invariant under the action of $G$. Therefore this case is impossible, since the disjoint union of these three curves can not be the orbit of curves under the action of the group $G=\mathbb Z_5$, since the factor-space of this orbit is the irreducible curve $M_1$.

Let $S$ be a cone and $\tau:\mathbb F_2\to S$ the resolution of the singular point of $S$, $\tau^{-1}(Sing(S))=R$. The group $Pic(\mathbb F_2)$  is generated by $R$ and $L$, $(R^2)_{\mathbb F_2}=-2$, $(R,L)_{\mathbb F_2}=1$, $(L^2)_{\mathbb F_2}=0$, where $L$ is a fibre of the projection of $\mathbb F_2$ to $R$. We have $\widetilde D\in |5R+10L|$. Let $\widetilde D_1\in |mR+nL|$, $m,n\geq 0$. Then $\widetilde D_2\in |(5-m)R+(10-n)L|$, $m\leq 5$ and $n\leq 10$ and, by (\ref{e1}), we have
$$(\widetilde D_1,\widetilde D_2)_{\mathbb F_2}= -2m(5-m)+n(5-m)+m(10-n)=m^2+5n-2mn=m(m-2n)+5n$$
and hence
\begin{equation}\label{e2}
m(m-2n)+5n=15,\qquad 0\leq m\leq 5,\quad 0\leq n\leq 10.\end{equation}
It is easy to check (if to note that $m(m-2n)$ is divisible by $5$) that if $(m,n)$ is an integer solution of (\ref{e2}), then $(m,n)\in \{ (0,3),(5,2)\}$, that is, either $\widetilde D_1\in |3L|$ if $(m,n)=(0,3)$, or $\widetilde D_2\in |8L|$ if $(m,n)=(5,2)$. As in the case when $S$ is smooth, it is easy to see that the case, when $S$ is a cone, is also impossible. \qed

\begin{claim} \label{cl6} Let $D=K_X+E \in\mathcal D_l$ be such that $D$ is not a bad element of $\mathcal D_l$ and let  $\dim H^0(X,\mathcal O_X(K_X+D+\alpha))=2$ for some  $\alpha\in Tor(X)$, $\alpha\neq 0$.  Then $\mathcal O_{C_j}(E)=\mathcal O_{C_j}$, where $C_j\in |K_X+\alpha|$.
\end{claim}
\proof  In the exact sequence
$$\begin{array}{c} 0\to H^0(X,\mathcal O_X(K_X+E))\to H^0(X,\mathcal O_X(K_X+C_j+E))\stackrel{i_{*}}{\rightarrow} \\ H^0(C_j,\mathcal O_{C_j}(K_{C_j}+E))\to   H^1(X,\mathcal O_X(K_X+E))\to \dots \end{array}$$
$i_{*}$ is an isomorphism, since $K_X+E$ is not a bad element. Therefore $\mathcal O_{C_j}(K_{C_j}+E))=\mathcal O_{C_j}(K_{C_j}))$ (and hence $\mathcal O_{C_j}(E)=\mathcal O_{C_j}$), since
$\dim H^0(X,\mathcal O_X(K_X+D+\alpha))=2$ and on the curve $C_j$ of genus $g(C_j)=2$ there is only one divisor $M$ of degree two, namely $K_{C_j}$, for which $\dim H^0(C_j,\mathcal O_{C_j}(M)=2$. \qed

\begin{claim} \label{cl7} Let $D=K_X+E \in\mathcal D_l$  and $D_1=K_X+E +\beta \in\mathcal D_l$, $\beta\in Tor(X)$, $\beta\neq 0$, be such that neither $D$, nor $D_1$ is a bad element of $\mathcal D_l$. Asume that  $$\dim H^0(X,\mathcal O_X(K_X+D+\alpha))=2$$ for some  $\alpha\in Tor(X)$, $\alpha\neq 0$.  Then $\dim H^0(X, \mathcal O_{X}(K_X+D_1+\alpha)\leq 1$.
\end{claim}
\proof Let $C_j\in |K_X+\alpha|$. By Claim \ref{cl6}, $\mathcal O_{C_j}(E)=\mathcal O_{C_j}(E+\beta)=\mathcal O_{C_j}$ if $$\dim H^0(X, \mathcal O_{X}(K_X+D_1+\alpha)\geq 2.$$ But, then $\mathcal O_{C_j}(\beta)=\mathcal O_{C_j}$ and we obtain the contradiction with Claim \ref{cl5}. \qed

Now, to complete the proof of Theorem \ref{T:example}, it suffices to apply Claims \ref{cl4} and \ref{cl7}. \qed\\

Let $(X_1,C_1,\mathcal F_1)$ and $(X_2,C_2.\mathcal F_2)$ be two pre-spectral data with  smooth surfaces $X_1$ and $X_2$.
We say that two pre-spectral data $(X_1,C_1,\mathcal F_1)$ and $(X_2,C_2,\mathcal F_2)$ are {\it strongly deformation equivalent} if there are a smooth complex threefold
$\mathcal X$, an effective divisor $\mathcal C$ and a sheaf $\mathcal F$ on $\mathcal X$, and a proper holomorphic surjection $p:\mathcal X\to \Delta=\{ z\in \mathbb C \, \, |\,\, |z|<1\}$ such that $X_1$ and $X_2$ are fibres of $p$ and $(X_t=p^{-1}(t), C_t=(X_t,\mathcal C)_{\mathcal X}, \mathcal F_t=\mathcal F_{|X_t})$ are pre-spectral data for  each $t\in \Delta$. Extend the strong deformation equivalence to an equivalence relation on the set of pre-spectral data with  smooth surfaces. Similar to, so called,  $"DIF=DEf"$-problems
in algebraic geometry (see, for example, \cite{Man}, \cite{K-H1}, \cite{K-H2}), we have a  $"DIF=DEf"$-problem for pre-spectral data if we fix a diffeomorphic type of a surface $X$. In particular, there is the following
\begin{que} 
How many deformation equivalence classes of pre-spectral data with smooth Godaux surfaces from Theorem {\rm \ref{T:example}} do exist{\rm ?}
\end{que}

\ifx\undefined\bysame
\newcommand{\bysame}{\leavevmode\hbox to3em{\hrulefill}\,}
\fi

\vspace{0.5cm}

\noindent Vik. S. Kulikov,  Steklov Mathematical Institute of Russian Academy of Sciences, Gubkina 8, Moscow, Russia
\newline e-mail
	$kulikov@mi.ras.ru$

\vspace{0.3cm}

\noindent A. Zheglov,  Lomonosov Moscow State  University, faculty
of mechanics and mathematics, department of differential geometry
and applications, Leninskie gory, GSP, Moscow, \nopagebreak 119899,
Russia
\\ \noindent e-mail
 $azheglov@mech.math.msu.su$

\end{document}